\documentclass[aap]{imsart}

\usepackage{amsmath}
\usepackage{amssymb}
\usepackage{amsfonts}
\usepackage{amsthm}
\usepackage{bbm}
\usepackage{color}
\usepackage{graphicx}
\usepackage[normalem]{ulem}
\usepackage{url}
\usepackage{enumitem}
\usepackage{subcaption}
\usepackage{verbatim}

\usepackage{nameref,hyperref}
\usepackage[capitalize]{cleveref}
\usepackage{algorithm}
\usepackage{algpseudocode}

\RequirePackage[numbers]{natbib}

\newcommand{\cC}{{\mathcal{C}}}
\newcommand{\cH}{{\mathcal{H}}}
\newcommand{\cF}{{\mathcal{F}}}
\newcommand{\cS}{{\mathcal{S}}}
\newcommand{\Ex}{{\mathbb{E}}}

\def\Ex{{\textup{E}}}

\def\Vol{{\textup{Vol}}}

\def\lrb{{\underline{R}_k}} 
\def\urb{{\overline{R}_k}} 
\def\hd{{h_{{\textup{d}}}}}
\def\hs{{h_{{\textup{s}}}}}

\def\Pr{\mathrm{Pr}}
\def\Tmix{\Delta_{{\tiny \textup{mix}}}}
\def\Thit{\Delta_{{\tiny \textup{hit}}}}

\theoremstyle{plain}
\newtheorem{theorem}{Theorem}
\newtheorem{lemma}{Lemma}[section]

\newtheorem{proposition}[lemma]{Proposition}

\theoremstyle{remark}

\newtheorem{remark}[lemma]{Remark}
\newtheorem{fact}[lemma]{Fact}

\makeatletter
\let\mymakefnmark\@makefnmark
\let\mythefnmark\@thefnmark

\newcommand{\restorefn}{\let\@makefnmark\mymakefnmark
\let\mythfnmakr\@thefnmark}
\makeatother

\begin{document}

\begin{frontmatter}
  \title{Dynamic Spatial Matching}

\runtitle{Dynamic Spatial Matching}

\begin{aug}
\author{\fnms{Yash} \snm{Kanoria}\ead[label=e1]{ykanoria@columbia.edu}}
\address{Graduate Business School, Columbia University, \printead{e1}}
\end{aug}

\begin{abstract}
Motivated by a variety of online matching platforms, we consider demand and supply units which are located
i.i.d. in $[0,1]^d$, and each demand unit needs to be matched with a supply unit. The goal is to
minimize the expected average distance between matched pairs (the ``cost'').
We model dynamic arrivals of one or both of demand and supply with uncertain locations of
future arrivals, and characterize the scaling behavior of the achievable cost in terms of system size (number of supply units), as a function of the
dimension $d$. Our achievability results are backed by  concrete matching algorithms. Across cases, we find that the platform can achieve cost (nearly) as low
as that achievable if the locations of future arrivals had been known beforehand.
Furthermore, in all cases except one, cost nearly as low in terms of scaling as the expected distance to
the nearest neighboring supply unit is achievable, i.e., the matching constraint does not cause an increase in cost either. The aberrant case is where only demand arrivals are dynamic, and $d=1$;
excess supply significantly reduces cost in this case. 
\end{abstract}

\begin{keyword}
\kwd{dynamic arrivals}
\kwd{matching algorithms}
\kwd{distance}
\kwd{spatial heterogeneity}
\kwd{length scales}
\end{keyword}
%

\end{frontmatter}


\section{Introduction}
\label{sec:intro}

Online matching platforms serve to match units (participants) on the two sides of the market, in contexts such as car or bike rides, lodging, dating, labor and organ exchanges.
This paper aims to understand and quantify the costs arising from stochastic spatial heterogeneity in dynamic matching as a function of market ``thickness'' and the dimensionality of the relevant ``space'' in which supply and demand live, and to identify near optimal control policies to perform such matching over time. Here space can refer to physical space, as may be relevant, e.g., in ridehailing, as well as to other relevant attributes, e.g., for a lodging platform like Airbnb or VRBO, in addition to physical location (which may be two-dimensional), there are other relevant attributes such as the  ``quality'' of the unit and the price. Supply units live in this attribute space and demand units can also be thought of as living in this same attribute space, with the location of a demand unit being interpreted as the ideal attribute vector desired.

Spatial matching costs (distances) in \emph{static} models with \emph{balanced} (i.e., equal) supply and demand, were precisely quantified as a function of ``thickness'' (density of points) and the dimensionality of the space in the seminal works of Ajtai, Komlos \& Tusnady \cite{ajtai1984optimal}, Leighton and Shor \cite{leighton1989tight}, and Talagrand \cite{talagrand1992matching}, among others. The problem studied in \cite{ajtai1984optimal,talagrand1992matching} is as follows; our  first model (which will serve as a benchmark) will be a generalization. There are $N$ \emph{supply} points and $N$ \emph{demand} points i.i.d. uniformly located in the $d$-dimensional unit hypercube $\cC \triangleq [0,1]^d$, and one wants to characterize the (expected) \emph{cost}, i.e., the average matching distance, of the minimum cost perfect matching between demand and supply, 
in terms of its scaling behavior with $N$. (For concreteness, we will consider Euclidean distance; all our results extend immediately to any norm equivalent to the Euclidean norm.) 
Now, the expected distance between a demand unit and the nearest supply  unit (henceforth, the \emph{nearest-neighbor-distance}) is $\Theta(1/N^{1/d})$, and this is clearly a lower bound on the expected average matching distance. \cite{ajtai1984optimal,talagrand1992matching} develop the following remarkable picture:
\begin{itemize}
  \item For $d \geq 3$, the expected minimum cost (average distance) is $\Theta(1/N^{1/d})$, which is the same scaling as the nearest-neighbor-distance  \cite{talagrand1992matching}. 
  \item For $d=2$, the expected minimum cost is $\Theta(\sqrt{\log N}/\sqrt{N})$ \cite{ajtai1984optimal}. 
       The nearest-neighbor-distance in this case is $\Theta(1/\sqrt{N})$, and we see that the expected minimum cost is a factor $\Theta(\sqrt{\log N})$ larger.
  \item For $d=1$, the expected minimum cost is $\Theta(1/\sqrt{N})$. This is much larger than the nearest-neighbor-distance, which is only $\Theta(1/N)$.
\begin{fact}[1-dimensional static random matching\footnote{There is a minimum cost matching as follows: order both demand and supply from left to right and then match units at the same rank in the ordering to each other. The result then follows from Donsker's theorem. (The expected cost of this matching is $\Theta(1/\sqrt{N})$ since the $f$-th fractile of demand and supply are separated by $\Theta(\sqrt{f(1-f)/N})$ distance in expectation, for all $f \in (0,1)$.)}]
\label{fact:AKT-1d-is-wrong}
Consider $N$ ``supply'' points and $N$ ``demand'' points, i.i.d. uniformly distributed in $[0,1]$. Then the minimum cost perfect matching between the two sets of points has expected average matching distance $\Theta(1/\sqrt{N})$, whereas the expected distance between a given demand unit and the closest supply unit is only $\Theta(1/N)$.
\end{fact}
\end{itemize}

Motivated by matching platforms, we ask whether and how the above picture for random spatial matching extends to the case of matching demand and supply under \emph{dynamic} arrivals. 
The key challenge under dynamic arrivals is uncertainty about the locations of future arrivals, which means that it is unclear how to choose matches, unlike in the static case where the minimum cost matching is easy to compute. Indeed, the dynamic spatial matching problem is a dynamic programming (control) problem with the system state being the locations of the supply units currently in the system (we will assume that demand units need to be matched immediately upon arrival), which is an infinite dimensional state without evident helpful structure.
Our task is to design a near optimal \emph{algorithm} (control policy) to conduct spatial matching over time. 
And of course we want to characterize the minimum achievable expected matching cost.

Locations of demand and supply are assumed to be i.i.d. uniformly random in the unit hypercube $\cC$ throughout, to enable us to focus on the role of spatial stochasticity, as in the literature on static random spatial matching. 
We study not just the case of equal demand and supply, but allow for excess supply. Our characterizations of cost as a function of the excess supply yield guidance on capacity planning.

We provide characterizations for three natural models:  
\begin{itemize}
\item The static model above but allowing for $M$ additional supply units. This setting builds on the known characterizations \cite{ajtai1984optimal,talagrand1992matching} for the balanced case and serves as a benchmark. 
\item A ``semi-dynamic'' model which is identical to the static model above, except that demand units arrive sequentially and must be matched immediately upon arrival.
\item A ``fully dynamic'' model where both supply and demand arrive over time, and there are $m$ supply units present in the system at any time. In each period, a demand unit arrives and must be matched immediately, the matched pair leaves, and then a supply unit arrives, restoring the number of supply units in the system to $m$. Arrival locations of both demand and supply are i.i.d. uniform in $\cC$.
\end{itemize}
We remark that these models are sequenced in increasing order of technical complexity, with the fully dynamic model being the most technically challenging by far.
\smallskip

{\bf Summary of contributions.} We characterize the scaling behavior of the minimum achievable cost in each of the three models for each $d$. In the static and semi-dynamic models the scaling parameters are $N$ and $M$ (where $M\leq N$ can scale arbitrarily with $N$), and in the fully dynamic model the scaling parameter is $m$. Our upper and lower bounds match up to a constant factor in all cases, except for the fully dynamic model with $d=1$, where the bounds match up to a $\log m$ factor.  We provide a matching algorithm to achieve near optimal expected cost in the dynamic models, which we call \emph{Hierarchical Greedy}. A key feature of the algorithm is analytical tractability; roughly, it separates out the matching cost at different length scales, and incurs a cost which is no larger (in terms of scaling) than the fundamental ``obstacle to matching'' at each length scale.

Our findings reveal that in each model, the minimum cost achievable increases with $d$ (by a polynomial factor in the scaling parameter(s)); this monotonicity is unsurprising  since higher $d$ means that compatibility is measured along more dimensions, which makes the problem harder (formally, the distance between two points in $d$ dimensions is lower bounded by the distance between their projections onto $d-1$ dimensions). Our results quantify this phenomenon.

In each model and for each $d$, we compare the minimum cost achievable with the relevant nearest-neighbor-distance\footnote{The nearest neighbor distance we compare against is optimistically computed in the dynamic models by assuming a uniform spatial distribution persists. In reality, the spatial distribution of existing units which arrived earlier depends on the past matching decisions and is non-uniform in general.}, and find that for $d \geq 2$ in all models, and also for $d=1$ in the fully dynamic model, cost only slightly more than the nearest-neighbor-distance is achievable.
The cost is significantly larger (i.e., by a polynomial factor as in Fact~\ref{fact:AKT-1d-is-wrong}) \emph{only} for $d=1$ and sublinear excess supply in the static and semi-dynamic models; henceforth we refer to this case as the aberrant case. 
Notably, a similar phenomenon does \emph{not} occur for $d=1$ in the fully dynamic model, where the nearest-neighbor-distance \emph{is} achievable up to a polylogarithmic factor. Our analysis identifies that this is because the fully dynamic matching in $d=1$ dimension conceptually resembles the problem of static matching in $d+1=2$ dimensions. The intuitive reason is that time acts like an additional spatial dimension in the fully dynamic model (this feature extends to $d>1$ as well).

Our results illuminate the impact of uncertainty about the future on matching cost: Happily, in both dynamic models and all cases uncertainty about the future does \emph{not} lead to a substantial (i.e., a polynomial factor) increase in achievable matching cost.

We briefly summarize our findings regarding the benefit from excess supply. For the aberrant case $d=1$ in the static and semi-dynamic models, adding substantial excess supply $M = \omega(\sqrt{N})$ leads to a substantial reduction in the matching cost, and this is the only case in which excess supply significantly reduces cost in these two models. The cost reduction from excess supply in the aberrant case comes \emph{without} an accompanying reduction in the nearest-neighbor-distance, and occurs, informally, because excess supply smooths over stochastic mismatches in the quantities of demand and supply at the larger length scales. In the fully dynamic model and for all $d \geq 1$, increasing excess supply (i.e., the volume $m$ of ``free'' supply present in the system at any time) reduces cost as a direct consequence of reduction in the nearest-neighbor-distance, which is perhaps unsurprising and an entirely different phenomenon than that which occurs in the aberrant case.

\smallskip

{\bf Organization of the paper.} Section~\ref{sec:static-matching}  provides characterizations for static spatial matching with excess supply, and serves as a benchmark in our study of dynamic spatial matching. Section~\ref{sec:semi-dynamic} provides our model and results for the semi-dynamic model. Section~\ref{sec:fully-dynamic} provides our model and results for the fully dynamic model. Section~\ref{sec:analysis} provides our algorithm for matching and its analysis, leading to the proofs of our main achievability results. The proofs of our lower bounds are provided in an appendix. Section~\ref{sec:capacity-planning} discusses an application of our results for the fully dynamic model to capacity planning for shared mobility systems. We conclude with a discussion of related work and open directions in Section~\ref{sec:discussion}.

\section{Static Matching}
\label{sec:static-matching}

We begin with a characterization of matching costs for \emph{static} spatial matching under excess supply, which extends the celebrated characterizations for the case of no excess supply \cite{ajtai1984optimal,talagrand1992matching}. Here, all supply and demand locations are simultaneously known, and so there is no algorithmic challenge; one simply chooses the minimum cost matching. The question is that of characterizing (the scaling behavior of) the minimum cost achievable. The cost in the static setting is of interest since it will be a lower bound on the minimum cost achievable in the corresponding dynamic setting we will investigate in the next section, where locations of future demand arrivals are unknown.

\smallskip

{\bf Model.} There are $N+M$ supply units and $N$ demand units located at i.i.d. uniformly random locations in $\cC \triangleq [0,1]^d$. We are interested in characterizing the expected cost of the minimum cost matching, among matchings which match all $N$ demand units, and where the cost of a matching is the average (Euclidean) distance between matched pairs. Note that $M$ supply units are left over at the end, hence we call $M$ the \emph{excess supply}. The excess supply may help reduce the expected cost and we aim to quantify this benefit; the minimum cost for the no excess supply case $M=0$ was known from previous work. We will permit any\footnote{An exception is $d=2$, where our theorem statement does not cover ``just-sublinear'' $M \in \mathcal{I} \triangleq (N^{1-\epsilon}, \epsilon N)$. Of course the small-$M$ upper bound on cost $C_2 \sqrt{\log N/N}$ extends immediately to $M \in \mathcal{I}$ (since the problem is easier for larger $M$), as does the large-$M$ lower bound $(1/C_2)\sqrt{1/N}$ (since our problem is harder for smaller $M$), i.e., there is a $\sqrt{\log N}$ factor gap between our lower and upper bounds for $M \in \mathcal{I}$.} $M \in \{0, 1, \dots, \Theta(N)\}$.

\begin{theorem}[Minimum cost for static matching]
\label{thm:static-matching}
For $d=1$, there exists a constant $C_1 \in [1, \infty)$ such that for all $N \geq 2$ we have
\begin{align}
\textup{Minimum achievable expected cost} \in
\left \{
\begin{array}{cc}
  \left [\frac{1}{C_1 \sqrt{N}}, \frac{C_1}{\sqrt{N}} \right ] & \textup{for all } M \leq \sqrt{N} \\[5pt]
  \left [\frac{1}{C_1 M}, \frac{C_1}{M} \right ] & \textup{for all } M \in (\sqrt{N}, N ]  
\end{array}
\right  . \, .
\end{align}

For $d=2$, for any $\epsilon \in (0,1)$, there exists a constant $C_2 = C_2(\epsilon) \in [1, \infty) $ such that for all $N \geq 2$, 
we have
\begin{align}
\textup{Minimum achievable expected cost} \in
\left \{
\begin{array}{ll}
\left [ (1/C_2)\sqrt{\frac{\log N}{N}}, C_2\sqrt{\frac{\log N}{N}} \right ] & \textup{for } M \leq N^{1- \epsilon}\\
\left [ (1/C_2)\sqrt{\frac{1}{N}}, C_2\sqrt{\frac{1}{N}} \right ] & \textup{for } M \in [\epsilon N, N]
\end{array}
\right . \, .
%
\end{align}

For $d\geq 3$, there exists a constant $C_d \in [1, \infty)$ such that for all $N \geq 2$ and $M \in \{0, 1, \dots, N\}$ we have
\begin{align}
\textup{Minimum achievable expected cost} \in \left [\frac{1}{C_d N^{1/d}}, \frac{C_d}{N^{1/d}} \right ] \, .
\end{align}
\end{theorem}



Theorem~\ref{thm:static-matching} quantifies how excess supply helps in the case $d=1$. For small excess supply $M \leq \sqrt{N}$,  the cost scales as $\Theta(1/\sqrt{N})$, as in the well known case $M=0$ where this scaling is an immediate consequence of Donsker's theorem. Notably, this cost scaling is nearly as large as that for $d=2$; ensuring proximity along two dimensions is not much harder than ensuring proximity in one dimension. For $d=1$ and substantial excess supply $M > \sqrt{N}$, the cost falls to $\Theta(1/M)$. In particular, for linear excess supply $M= \Theta(N)$, the cost has the same scaling as the nearest-neighbor-distance $\Theta(1/N)$.

In the case $d=2$, the minimum cost for sublinear excess supply $M \leq N^{1-\epsilon}$ scales as $\Theta\big (\sqrt{\tfrac{\log N}{N}} \big)$, as in the Ajtai-Komlos-Tusnady theorem \cite{ajtai1984optimal} concerning the case $M=0$. Linear excess supply $M = \Theta(N)$ eliminates the $\sqrt{\log N}$ factor and allows to achieve the nearest-neighbor-distance scaling $\Theta(1/\sqrt{N})$.

In the cases $d\geq 3$, even with $M=0$ the cost is already as small as the nearest-neighbor-distance $\Theta(1/N^{1/d})$ as known from the work of Talagrand \cite{talagrand1992matching}, and this clearly extends to any $M \leq N$.

\section{Semi-dynamic model}
\label{sec:semi-dynamic}

We now consider a dynamic version of the same model, where supply units are present beforehand but where demand units arrive sequentially and need to be matched immediately upon arrival. Uncertainty about the locations of future demand arrivals poses a challenge, and leads to two questions: How much larger is the minimum achievable cost relative to the static model, i.e., how large is the burden imposed by uncertainty about the future? And how should we assign a supply unit to each demand unit to achieve near optimal cost?

\smallskip
{\bf Model.} There are $N+M$ supply units present at the outset,  located at i.i.d. uniformly random locations in $\cC=[0,1]^d$. No further supply units will arrive thereafter and supply units do not change their location. The following occurs at each period $t$ for $t = 1, 2, \dots, N$.
\begin{itemize}
  \item A demand unit arrives at a uniformly random location in $\cC$ (independent of the history so far).
  \item It must be matched immediately to one of the supply units resulting in a cost equal to the distance  between them. The matched pair leaves.
\end{itemize}
The problem is again to minimize the expected average distance between matched pairs.

Such a model may be a reasonable approximation of reality in certain real world settings, for example, for a lodging platform, it may provide a way to think about the problem of dynamically managing inventory of supply for a given geographical area and date(s) (e.g., a particular weekend in the Lake Tahoe area), if suppliers post their inventory in advance but demand arises closer to the date.

The following main result provides a characterization of the minimum average matching cost achievable for all $d$ and all $M$. Note that all lower bounds are ``inherited'' from the static matching model, since uncertainty about the location of future demand only makes things harder, and the two models are identical in all other respects. The content of the theorem lies in the upper bounds.

\begin{theorem}
All bounds on the expected average match distance in Theorem~\ref{thm:static-matching} hold verbatim for the semi-dynamic model.\footnote{In other words, there exist constants $(C_d)_{d=2}^\infty$ and $C_1(\epsilon)$ such that the bounds apply to both the static model and the semi-dynamic model.}
%
\label{thm:semi-dynamic}
\end{theorem}

The theorem says that the scaling behavior of the minimum cost achievable in each case is entirely unaffected by uncertainty about the locations of future demands, i.e., one can do nearly as well in the semi-dynamic model as in the static model.
The algorithm leading to our upper bounds in each case is the Hierarchical Greedy algorithm which we provide in Section~\ref{subsec:HG}; with the exception of the $d=2$ and $M \leq N^{1-\epsilon}$ where we use a result of \cite{holden2021gravitational} to show achievability. Theorem~\ref{thm:semi-dynamic} is proved in Section~\ref{subsec:proof-SD-achievability}. Our analysis illuminates how and why excess supply helps for $d=1$ (by mitigating the obstacle to matching at the larger length scales) but does not help for larger dimensions $d\geq 2$ (because the nearest-neighbor-distance dominates the cost anyway). 
See Figure~\ref{fig:SD-and-static-thms} for a visual representation of Theorem~\ref{thm:semi-dynamic} (and Theorem~\ref{thm:static-matching}).

\begin{figure}[tb]
  \centering
  \includegraphics[width=0.64\textwidth]{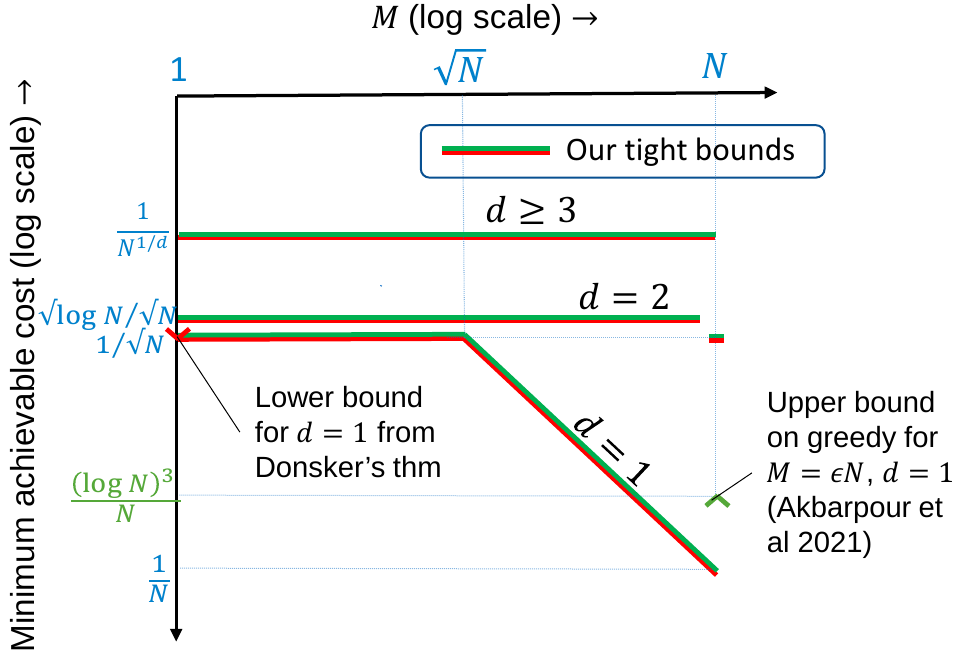}
  \caption{Minimum achievable cost versus $M$ for different $d$, from Theorem~\ref{thm:semi-dynamic} (and Theorem~\ref{thm:static-matching})}\label{fig:SD-and-static-thms}
\end{figure}

Compare our tight bounds for all $d$ and all $M$ with the analysis of \cite{akbarpour2021value} which studies the special case $d=1$ and $M = \Theta(N)$ and establishes a  upper bound of $O(\log^3 N/N)$ on the average matching cost under the standard greedy algorithm (our results imply that the minimum cost in this case is in fact $\Theta(1/N)$). 


\section{Fully dynamic model} 
\label{sec:fully-dynamic}

Finally, we consider a model where both supply and demand units arrive over time. This is the most technically challenging setting. We will adopt the simplest model which captures the challenge of spatial heterogeneity, namely, the tension between achieving a small matching cost \emph{now}, and maintaining a supply distribution (state) which is well spread out and hence will facilitate cheap matching in \emph{future}. Demand units will need to be matched immediately upon arrival. 
Our model will be set up so that there are exactly $m$ supply units in the system at any time. (While our analysis can be extended to a setting where supply and demand arrivals follow independent Poisson processes, with supply units arriving slightly faster and abandoning at a certain rate, and a cost for making demand wait, such a formulation adds unilluminating technical complexity so we avoid it.) We will formulate a discrete time  average cost minimization problem over a finite (or infinite) horizon, and ask how the average matching cost per pair scales with $m$ as a function of the dimension $d$.

\smallskip
{\bf Model.}
Initially there are $m$ supply units at arbitrary locations in $\cC=[0,1]^d$. At each period $t=1, 2, \dots, N$, the following sequence of events occurs (instantaneously):
\begin{itemize}
  \item A demand unit arrives at a uniformly random location in $\cC$ (independent of the past) and requires to be matched immediately.
  \item The system operator chooses one of the $m$ supply units in the system to match the demand unit to. The matched pair leaves and the period cost is equal to the Euclidean distance between them. 
  \item A supply unit arrives at a uniformly random location (independent of the past). This restores the number of supply units present in the system to $m$.
\end{itemize}

The goal is to minimize the expected average cost per period over some horizon $N$. Note that this model has just one key model primitive $m$, and as before, the question will be how the minimum achievable cost scales with $m$ (for large enough $N$). 

The following lower bound on the average cost per period is immediate from considering the expected distance between the next demand arrival and the closest supply unit. (This lower bound applies even if the system operator is permitted to costlessly reposition the $m$ supply units arbitrarily prior to each time period.) The proof is in the appendix.
\begin{proposition}[Nearest-neighbor-distance lower bound]
\label{prop:closest-nhbr-lower-bound}
For each $d \geq 1$ there exists $\epsilon = \epsilon(d)> 0$ such that for all $m \geq 2$, the expected match cost in each period for any policy is at least $\epsilon/m^{1/d}$. This holds for arbitrary initial locations of supply units and any horizon $N \geq 1$.
\end{proposition}


Our first result shows that the lower bound in Proposition~\ref{prop:closest-nhbr-lower-bound} is \emph{not} achievable for $d = 1$. In other words, we show a tighter lower bound for the one-dimension case. The proof is in the appendix.
\begin{theorem}[Tighter lower bound for $d=1$]
\label{thm:1d-impossibility}
For $d=1$, there exists a constant $\epsilon> 0$, such that for all $m \geq 2$ and any horizon $N \geq m^2$, the expected average cost per period for any policy is at least $\epsilon \log m/m$. This holds for arbitrary initial locations of supply units.
\end{theorem}
The proof proceeds by formally developing an  analogy with two-dimensional static random matching \cite{ajtai1984optimal}. In our analogy time acts as an additional spatial dimension. In particular, this connection assures us that any policy developed for the fully dynamic case implies a matching algorithm and corresponding match cost for two-dimensional static random matching. But, the result of Ajtai, Komlos and Tusnady \cite{ajtai1984optimal} tells us that two-dimensional static random matching incurs a expected match distance which is $\Omega (\sqrt{N \log N})$, as captured in our Theorem~\ref{thm:static-matching}. We deduce the lower bound in Theorem~\ref{thm:1d-impossibility}.

This negative result raises the question of whether the nearest-neighbor based lower bound is unachievable also for $d \geq 2$. How does the optimal average matching cost scale with $m$ for each $d$? We answer this question definitively for $d \geq 2$, and up to a $\log m$ factor for $d=1$.

\begin{theorem}[Achievability]
\label{thm:MG-performance}
For each $d \geq 2$, there is a constant $C_d < \infty$, such that for all $m \geq 2$, there is a matching policy which achieves an expected average cost per match of at most $C_d/m^{1/d}$, both in steady state and for any finite horizon $N \geq C_d m^{1.01+1/d}$. For $d =1$, there is a constant $C_1 < \infty$, such that for all $m \geq 2$, there is a matching policy which achieves an expected average cost per match of at most $C_1(\log m)^2/m$, both in steady state and for any finite horizon $N \geq C_1 m^3$.
For each $d\geq 1$, if, before period 1, the $m$ supply units are at evenly spread locations,\footnote{As an example of ``evenly spread locations'', the $m$ supply units can be located at the lattice points of a grid $\{(i\epsilon, j\epsilon): i,j \in \{0, 1, \dots, m^{1/d}-1 \}, \epsilon = 1/(m^{1/d}-1)\}$.} the aforementioned upper bounds on the expected cost hold for arbitrary finite horizon $N \geq 1$.
\end{theorem}

To establish Theorem~\ref{thm:MG-performance}, we use a more sophisticated version of our Hierarchical Greedy algorithm (Section~\ref{subsec:HG}) and analyze its performance (Section~\ref{subsec:proof-FD-achievability}). 

We are pleased to have obtained these sharp results in what appeared to be a challenging  setting to analyze. This is an average cost minimization problem where the state of the system (the locations of $m$ supply units) belongs to a high-dimensional continuum $\cC^m = [0,1]^{dm}$ with $d$-coordinates per supply unit (of course there is invariance under permutations of supply units) and an underlying metric (here, the Euclidean distance). Discretizing $\cC$ nevertheless requires us to consider $\Omega(m)$ discrete locations arranged in a $d$-dimensional grid in order to study distances of order $1/m^{1/d}$, and it is unclear how to solve a control problem of match-distance-minimization in such a state space. The fully dynamic setting appears substantially more challenging than the semi-dynamic setting studied earlier, and indeed, to obtain our achievability results we need a more involved version of Hierarchical Greedy with carefully chosen lower bound constraints on the supply maintained at each ``level'' in the hierarchy. Recall that a sharp analysis of the standard greedy algorithm is challenging even in the semi-dynamic setting, and analyzing greedy in the fully dynamic setting may be even harder. 


We use Theorem~\ref{thm:MG-performance} to derive capacity planning prescriptions for shared mobility systems like ridehailing in Section~\ref{sec:capacity-planning}.

\section{Matching algorithm and proofs of achievability results}
\label{sec:analysis}

In Section~\ref{subsec:HG} we provide our Hierarchical Greedy algorithm for dynamic matching which underlies all our achievability results except for the semi-dynamic model with $d=2$. In Section~\ref{subsec:proof-SD-achievability} we show achievability in the semi-dynamic model (Theorem~\ref{thm:semi-dynamic}). In Section~\ref{subsec:proof-FD-achievability}, we show achievability in the fully dynamic model (Theorem~\ref{thm:MG-performance}).

\subsection{The Hierarchical Greedy algorithm}
\label{subsec:HG}
Fix dimension $d$. We partition the hypercube $\cC=[0,1]^d$ into a sequence of successively refined partitions $\cH_{\ell_0} \triangleq \{\cC\}, \cH_{\ell_0-1}, \dots, \cH_0$, where the non-negative integer $\ell_0$ will be chosen later.
The partition $\cH_{\ell_0-1}$ is obtained by dividing $\cC$ into $2^d$ smaller \emph{child} hypercubes by cutting $\cC$ midway along each dimension. For example for $d=2$, the four children of $\cC$ are $\cH_{\ell_0-1} = \{ [i/2,(i+1)/2]\times[j/2,(j+1)/2] \; \forall \; (i, j) \in \{0,1\}^2\}$. (The points on the boundary between the child hypercubes can be included in any one of the adjacent hypercubes; the boundary set plays no role because it has Lebesgue measure zero.) We call the set $\cH_{\ell_0-1}$ of hypercubes produced the \emph{level-$(l_0-1)$ partition}.  We repeat this process of subdivision iteratively to produce successively refined partitions of $\cC$, subdividing each hypercube in $\cH_k$ into \emph{child} hypercubes by cutting it midway along each dimension to produce a refined partition $\cH_{k-1}$ for $k = \ell_0,\ell_0-1, \dots,1$. Figure~\ref{fig:HG-partitions} provides an illustration for the case $d=2$. Our choice of $\ell_0$ for the fully dynamic model will lead to the following number of \emph{leaf} hypercubes in the finest partition $\cH_{0}$:
\begin{align}
  2^{d \ell_0} = \left \{
  \begin{array}{ll}
  \Theta(m) & \textup{for } d \geq 2\\
  \Theta(m/\log m) & \textup{for } d=1
  \end{array} \, .\right .
\end{align}
Later, for the semi-dynamic model we will use $\ell_0$ such that $2^{d\ell_0} = \Theta(N)$ for all $d \geq 1$.

\begin{figure}[tb]
  \centering
  \includegraphics[width=0.5\textwidth]{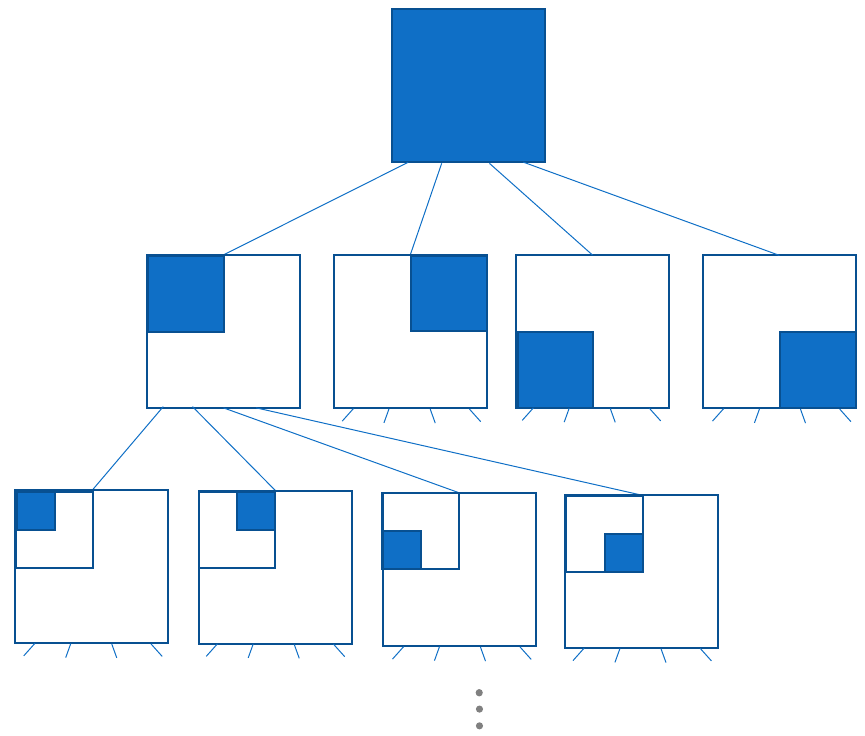}
  \put(-250,153){$\cH_{\ell_0}$}
   \put(-250,93){$\cH_{\ell_0-1}$}
    \put(-250,37){$\cH_{\ell_0-2}$}
  \caption{Hypercube partitions for Hierarchical Greedy for $d=2$}\label{fig:HG-partitions}
\end{figure}

Note that the set $\cH_k$ of hypercubes at level $k \in \{0, 1, 2, \dots, \ell_0 \}$ constitutes a partition of $\cC$ into $2^{(\ell_0-k) d}$ hypercubes, each of side-length $2^{-(\ell_0-k)}$  and volume $2^{-(\ell_0-k) d}$.
Also note the hypercubes defined by the above refinement process are organized (by definition) in a $2^d$-ary tree with height $\ell_0$. The root of the tree (whose height in the tree is $\ell_0$) is the original hypercube $\cC$, the nodes at height $\ell$ in the tree are the hypercubes in the level-$\ell$ partition $\cH_\ell$, and our notion of child hypercube uniquely specifies the child nodes of each node in the tree. Having established this tree organization, we will find it pedagogically useful to also deploy the standard terminology of \emph{parent}, \emph{ancestor} and \emph{descendant} (hypercube) in addition to the notion of child (hypercube).

\begin{algorithm}[tb]
\caption{Hierarchical Greedy (HG) for dynamic spatial matching}
\label{alg:HG}

\begin{flushleft}
INPUTS:\\
The number of dimensions $d$ of the unit cube $\cC=[0,1]^d$.\\
The number of initial supply units and their locations in $\cC$.\\
The number of levels of partition refinements $\ell_0+1$\\
The minimum supply to be maintained for hypercubes at each level $(\gamma_k)_{k=0}^{\ell_0}\,$, which must satisfy ${ \gamma_{\ell_0}\in\big[\, 0\,,\, \textup{(\# supply units)}\,\big)}$ as well as
\begin{align}
\gamma_{k}  \in \big [0, 2^{-d} \gamma_{k+1} \big] \qquad \qquad \forall k = 0, 1, \dots, \ell_0-1 \, .
\label{eq:gamma_k-LB}
\end{align}
In each period: The location of the demand arrival and then, if this is the fully dynamic model, the location of the supply arrival.
\smallskip

OUTPUT:\\
In each period: The existing supply unit to match the arrived demand with.
\end{flushleft}
\smallskip

\begin{algorithmic}[1]
\State Compute the (initial) number of supply units $n_h$ for all hypercubes $h$ in $\cH_0 \cup \cH_1 \cup \dots \cup \cH_{l_0}$
\For{period $t \geq 1$}
    \State $n_h(t) \gets n_h \ \ \forall h \in \cH_0 \cup \cH_1 \cup \dots \cup \cH_{l_0}$ \Comment{$n_h(t)$ denotes supply unit count at the beginning of period $t$}
    \State A demand unit arrives to some leaf hypercube $\hd = \hd(t) \in \cH_0$
    \State $\cS \gets \{k : 0 \leq k \leq \ell_0, n_{A_k(\hd)} \leq \gamma_{k} \}$ \Comment{The set of undersupplied ancestors of $h$}
    \If {$\cS = \emptyset$} \Comment{Determine the level $\ell$ at which the demand will be matched.}
        \State $\ell \gets 0$
    \Else
        \State $\ell \gets 1+ \max\{k: k \in \cS \}$ \Comment{Preserve supply of all undersupplied ancestors}
    \EndIf
    \State $h' \gets A_{\ell}(\hd)$ \Comment{Start from the level $\ell$ ancestor of $h$}
    \For{$k \gets \ell$ to $1$} \Comment{Iteratively pick the best-supplied child of $h'$, to reach a leaf}
        \State $h' \gets \arg \max \{n_h(t): h $ is a child of $h' \}$
    \EndFor
    \State Match the demand unit to an arbitrary supply unit in the leaf $h'$
    \State $n_{h'} \gets n_{h'} - 1$ at $h'$ and all its ancestors
    \If {fully dynamic setting}
        \State A supply unit arrives to some leaf hypercube $\hs = \hs(t) \in \cH_0$
        \State $n_{h} \gets n_{h} + 1$ at all $h \in \{A_k(\hs): 0 \leq k \leq \ell_0 \}$ \Comment{Increment $n_h$ at $\hs$ and all its ancestors}
    \EndIf
\EndFor

\end{algorithmic}
\end{algorithm}

%

Our Hierarchical Greedy (HG) will use the finest partition $\cH_{0}$ as a discretization of $\cC$, i.e., it will not distinguish between different supply units in the same leaf hypercube. We say that a demand unit has arrived to a leaf hypercube $\hd \in \cH_0$ if the location of that demand unit belongs to $\hd$, and similarly for supply units.  We say that a demand unit is matched at level $\ell$ if
\begin{align}
  \ell = \min\{ k : \exists h \in \cH_k \textup{ s.t. the demand unit and its matched supply unit are both in } h \} \, ,
  \label{eq:match_at_level_l}
\end{align}
i.e., the supply unit it is matched to belongs to the same hypercube at level $\ell$, but this is not true for any $\ell' < \ell$. Denote the number of supply units in each hypercube $h \in \cH_k$ prior to period $t$ by $n_h(t)$. For a leaf $h \in \cH_0$, denote by $A_k(h)$ the level $k$ ancestor of $h$, including $A_0(h) = h$. Hierarchical Greedy is specified in Algorithm~\ref{alg:HG}. 

In the semi-dynamic model, we will deploy Hierarchical greedy with no minimum supply requirements $\gamma_k = 0$ for all $0 \leq k\leq \ell_0$, and so $\ell$ will simply be the lowest level ancestor which has at least one supply unit. (There is always such a level as long as $M \geq 0$ since the number of supply units $N+M$ at the beginning exceeds the number of periods $N$.) Moreover, our analysis will not depend on which leaf descendant of that ancestor the supply unit is taken from. In contrast, in the more challenging fully dynamic model, both the $\gamma_k$s as well as the leaf to take supply from will need to be carefully chosen.

\begin{lemma}[Correctness of Hierarchical Greedy in the fully dynamic model]
We always have $\ell_0 \notin \cS$, where $\cS$ in defined in line 5 of Hierarchical Greedy. Hence $\ell \leq \ell_0$ in line 9 of Hierarchical Greedy. Furthermore, the ancestor  at level $\ell$ of the demand unit $A_{\ell}(\hd)$ has at least one supply unit available, and hence so does the leaf $h'$ invoked in line 15 of Hierarchical. 
\end{lemma}
\begin{proof}
The fact that the root hypercube is never undersupplied $\ell_0 \notin \cS$ is immediate from the requirement $\gamma_{\ell_0} < \textup{\# supply units}$, the definition of $\cS$, and the fact that the total number of supply units in the system remains the same over time. As a result $\max\{k: k \in \cS \} \leq  \ell_0-1$ and hence $\ell = 1+ \max\{k: k \in \cS \} \leq \ell_0$ in line 9 of HG.

Since $\gamma_k \geq 0$ for all levels $k$, we infer that $A_{\ell}(\hd)$ does have at least one supply unit, since it is not undersupplied, by definition of $\cS$ and $\ell$. 
Now, in lines 12--14, the algorithm iteratively picks the best supplied child. As a result, each $h'$ selected in line 13 has at least one supply unit. In particular, the leaf $h’$ invoked in line 15 also has at least one supply unit. 
\end{proof}

The idea of the HG algorithm for the fully dynamic model is twofold:
\begin{itemize}
  \item (Lines 5--10) We match each demand unit at the lowest level -- i.e., the smallest $k$ -- possible, subject to never using a supply unit from any hypercube $h \in \cH_k$ which currently has $\gamma_k$ or fewer supply units, for any $k \leq \ell_0$. The minimum targets $\gamma_k$ will be carefully chosen to balance the risk of supply depletion at different levels.
  \item (Lines 11--15) When matching at level $k>0$ there is a choice between the leaf nodes which share the same level-$k$ ancestor. We make this choice by -- starting with the ancestor at level $k$ -- iteratively picking the child hypercube which currently has the largest $n_h$. 
\end{itemize}


\subsection{Semi-dynamic model: Proof of Theorem~\ref{thm:semi-dynamic}}
\label{subsec:proof-SD-achievability}

In the semi-dynamic model, we use the Hierarchical Greedy algorithm with $\gamma_k=0$ for all $k$, i.e., we simply match each arriving demand unit to its lowest-level ancestor hypercube which has non-zero supply units remaining.\footnote{Note that, here, the set $S$ takes the form $S=\{ 1, 2, \dots, \ell\}$ for some $\ell$, since if some ancestor of the leaf where demand just arrived has non-zero supply units, all higher levels ancestors also have non-zero supply units.} Recall that the initial number of supply units is $N+M$ and these are depleted over time as the $N$ demand units arrive sequentially.

\subsubsection{Intuition: Obstacles to matching at different length scales}
\label{subsec:intuition-obstacles}

We start by providing some intuition regarding the obstacles to matching at each length scale, and why Hierarchical Greedy with $\gamma_k=0$ for all $k$ is a near optimal algorithm for the semi-dynamic model in light of these obstacles. The phrase ``obstacle to matching at length scale $L$'' refers to a lower bound on matching costs which holds due to discrepancy between demand and supply at length scale $L$. The development in this subsubsection is informal and may be skipped by the reader who prefers rigor. Our proof in the next subsubsection reveals clearly the formal manifestation of the obstacles to matching.

Consider random static matching in dimension $d$ with no excess supply $M=0$, and a sub-hypercube $h$ of the unit hypercube $\cC$ of side length\footnote{Here $1/N^{1/d}$ is the scaling of the nearest-neighbor-distance; we impose $L \geq 1/N^{1/d}$ to ensure that $h$ contains $\Omega(1)$ demand (supply) units in expectation.} $L \in [1/N^{1/d}, 1/2]$. The sub-hypercube has volume $L^d$ and ``surface area'' $\Theta(L^{d-1})$. 
The realized number of demand (supply) units in the hypercube is a Binomial($N, \textup{Volume}(h)$) random variable, and hence has mean $NL^d$ and standard deviation $\Theta(\sqrt{N L^d})$. As a result the expected discrepancy between the number of supply units and the number of demand units in the hypercube is $\Theta(\sqrt{N L^d})$. Suppose, for instance, the realized number of supply units is smaller than the number of demand units. As a result, one would need to match some of demand units in the hypercube to supply units outside the sub-hypercube $h$. How far outside $h$? If one considers points within distance $\Delta$ from $h$, this incorporates a region outside $h$ of volume approximately the surface area times $\Delta$, i.e., $\Theta(\Delta L^{d-1})$. For this outside region to include $\Theta(\sqrt{N L^d})$ additional supply units, we need $N \cdot \Delta L^{d-1}  \sim \sqrt{N L^d} \Rightarrow \Delta \sim L^{1-d/2}/\sqrt{N}$. With a little bit of work, one can show a lower bound of this order $\Delta_L \sim L^{1-d/2}/\sqrt{N}$ on the expected average matching distance achievable for all $L \in [1/N^{1/d}, 1/2]$. (We formalize this lower bound for $d=1$ in the proof of Theorem~\ref{thm:static-matching}.) We observe that for $d=1$, $\Delta_L$ is monotone increasing in $L$ and hence the largest obstacle occurs at the ``macro scale'' $L\sim 1$, for $d=2$, the obstacles at all scales $L \in [1/\sqrt{N}, 1/2]$ have the same scaling with $N$, and for $d\geq 3$, $\Delta_L$ is monotone decreasing in $L$ and hence the largest obstacle to matching occurs at the ``micro scale'' corresponding to the nearest-neighbor-distance, i.e., $L \sim 1/N^{1/d}$. See Figure~\ref{subfig:obstacles-M=0} for a visual representation of $\Delta_L$ versus $L$ for different $d$.

\begin{figure}[tb]
     \centering
     \begin{subfigure}[b]{0.48\textwidth}
         \centering
         \includegraphics[width=\textwidth]{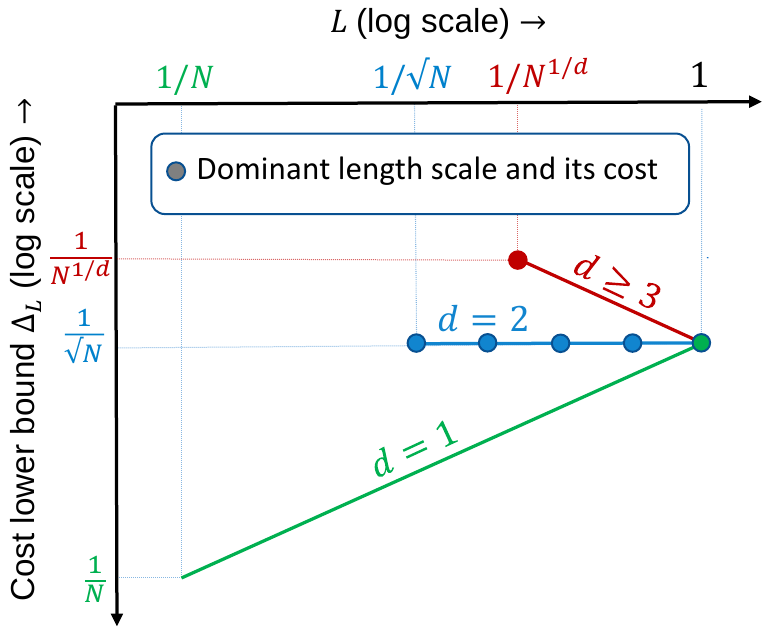}
         \caption{$M=0$}
         \label{subfig:obstacles-M=0}
     \end{subfigure}
         \hfill
     \begin{subfigure}[b]{0.48\textwidth}
         \centering
         \includegraphics[width=\textwidth]{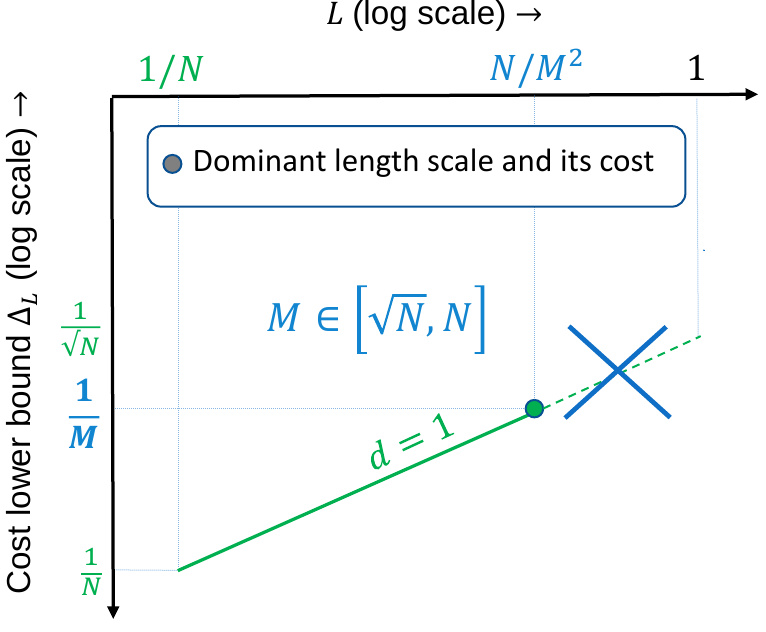}
         \caption{$d=1$ and $M \gtrsim \sqrt{N}$}
         \label{subfig:obstacles-d=1-Mlarge}
     \end{subfigure}
        \caption{Obstacles to matching at different length scales}
        \label{fig:obstacles}
\end{figure}

How does excess supply $M > 0$ modify the above heuristic picture? 
With excess supply $M$, the density of supply units is now $N+M$. The expected excess supply in the sub-hypercube $h$ grows linearly in the volume as $M \times \textup{Volume}(h)$ whereas the stochastic fluctuation in the total demand in $h$ grows only as the square root of the volume as $\sim \sqrt{N \times \textup{Volume}(h)}$.  Comparing the two aforementioned quantities, we deduce that for longer length scales $L \geq L_* = (N/M^2)^{1/d}$, the expected excess supply is larger than the stochastic variability $M \times \textup{Volume}(h) \geq \sqrt{N \times \textup{Volume}(h)}$ so one typically does not need to match demand in $h$ to supply outside $h$, i.e., \emph{excess supply (nearly) eliminates the obstacle to matching at longer length scales} $\Delta_L \sim 0$ for $L \gtrsim L_* = (N/M^2)^{1/d}$.
In contrast, for $L < L_*$ the expected excess supply is smaller than the stochastic variability $M \times \textup{Volume}(h) < \sqrt{N \times \textup{Volume}(h)}$ and one can again show an a $\Delta_L \sim L^{1-d/2}/\sqrt{N}$ lower bound on the expected match distance achievable, i.e., the obstacle to matching at the smaller length scales remains as in the no-excess-supply case.
For $d \geq 2$, we saw that there is a largest obstacle to matching at the micro scale $L \sim 1/N^{1/d}$, and excess supply $M \leq N$ cannot overcome this obstacle, hence excess supply does not reduce the achievable cost significantly.
In contrast, for $d = 1$, the obstacles to matching are larger at longer length scales for $M=0$, and excess supply $M \gtrsim \sqrt{N}$ helps significantly by (nearly) eliminating these obstacles for length scales $L \gtrsim L_*=N/M^2$. See Figure~\ref{subfig:obstacles-d=1-Mlarge} for a visual representation of $\Delta_L$ versus $L$ for $d = 1$ and $M \in [\sqrt{N}, N]$.

Coming to the Hierarchical Greedy algorithm, a key attractive feature of this algorithm is that the expected total cost that it incurs due to matches at length scale $L \in \{1,1/2, 1/4, \dots\}$ is only as large (in terms of scaling) as the obstacle to matching $\Delta_L$ at the corresponding length scale times the total number of matches $N$. Remark~\ref{rem:HG-aligns-with-deltaL} below notes the emergence of this property in our formal analysis of Hierarchical Greedy. (One may expect a similar property to hold also for the standard greedy algorithm, but the state of the system under the standard greedy algorithm is hard to keep track of analytically \cite{akbarpour2021value}.) 
As a result the total cost of hierarchical greedy is the sum of the obstacles to matching at length scales $1/2^{\ell_0d}, 1/2^{\ell_0d -1}, \dots, 1$. For $d=1$ this series is geometrically increasing (by factor $\sqrt{2}$), and for $d \geq 3$ the series is geometrically decaying (by factor $2^{1-d/2}$). For any geometric series with factor bounded away from $1$, the largest term and the sum are within a constant factor of each other, so Hierarchical Greedy achieves the same cost scaling as the largest obstacle. For $d=2$, there is a $\log N$ factor gap for $M = O(N^{0.99})$ because there are $\log N$ obstacles of roughly the same size and so the cost of Hierarchical Greedy would be the largest obstacle times $\log N$; in this case we use a different algorithm to improve on this performance by a $\sqrt{\log N}$ factor (roughly, the alternate algorithm takes advantage of cancellation in discrepancies between demand and supply at different length scales), which yields a tight bound. 

%

\subsubsection{Formal proof}
We now provide a formal proof of Theorem~\ref{thm:semi-dynamic} characterizing the minimum achievable cost in the semi-dynamic model.

We denote the number of demand units in $h$ which arise during up to and including period $t$ by $\hat{n}_h(t)$. As in the theorem, we assume $M \leq N$ throughout. We fix the number of refined partitions $\ell_0$ as the largest integer satisfying $2^{\ell_0 d} \leq N$.

\begin{lemma}
Fix $t \in \{0, 1, \dots, N-1\}$. 
Define $k_*  \triangleq 1+\inf\big \{k \in \mathbb{Z}: M+N-t \geq \sqrt{19 N 2^{d(\ell_0-k)}}\big \}$. Note that $k_* \geq 0$.
If $k_* \leq \ell_0$, the probability that a given hypercube at level $k \geq k_*+1$ has at least one child which has seen (weakly) more demand up to period $t$ than its initial supply is bounded above
by $2\nu^{k-k_*}$, where $\nu \triangleq (2/e^2)^{d} < 1/10$ for all $d\geq 1$. 
\label{lemma:prob-hypercube-in-trouble}
\end{lemma}

\begin{proof}
We first argue that $k_* \geq 0$. Note that for $k \leq -1$, $\sqrt{19 N 2^{d(\ell_0-k)}} > \sqrt{19 N \cdot N} > 2N \geq M + N \geq M+N-t $, where we used the definition of $\ell_0$ and $M \leq N$. It follows that $k \leq -1$ are not a part of $\{k \in \mathbb{Z}: M+N-t \geq \sqrt{19 N 2^{d(\ell_0-k)}}\big \}$, and hence $k_*\geq 0$.

We begin by bounding from above the probability that a hypercube at level $k'\geq k_*$ has seen (weakly) more demand through $t$ than its initial supply. 
Consider a hypercube $h'$ at level $k'$. By Bernstein's inequality
\begin{align*}
  &\Pr (n_{h'}(1) \leq 2^{d(k' - \ell_0)}(t+N+M)/2) \\
  &\leq \exp \left ( - \frac{(1/2)((N+M-t)2^{d(k' - \ell_0)}/2)^2}{(N+M)2^{d(k' - \ell_0)} + (N+M-t)2^{d(k' - \ell_0)}/6} \right ) \\
  &\leq \exp \left ( - \frac{(N+M-t)^22^{d(k' - \ell_0)}}{8(N+M) + 4(N+M-t)/3} \right ) \leq \exp \left ( - \frac{(N+M-t)^22^{d(k' - \ell_0)}}{19N} \right ) \, ,
\end{align*}
using $M \leq N$.
Similarly, for $\hat{n}_{h'}$ we obtain
\begin{align*}
  \Pr ( \hat{n}_{h'}(t) \geq 2^{d(k' - \ell_0)}(t+N+M)/2) &\leq \exp \left ( - \frac{(N+M-t)^22^{d(k' - \ell_0)}}{11N} \right ) \, .
\end{align*}
It follows from a union bound that
\begin{align}
\label{eq:prob-running-out-t}
&\Pr(\hat{n}_{h'}(t) \geq n_{h'}(1)) \\
&\leq \Pr \big( \; (n_h(1) \leq 2^{d(k' - \ell_0)}(t+N+M)/2) \ \textup{OR} \ (\hat{n}_h(t) \geq 2^{d(k' - \ell_0)}(t+N+M)/2)\; \big) \, \nonumber\\
&\leq 2 \exp \left ( - \frac{(N+M-t)^22^{d(k' - \ell_0)}}{19N} \right ) \, \nonumber\\
&\leq 2 \exp(-2^{d(k' - k_*+1)}) \, , \qquad \textup{using $k' \geq k_*$ and the definition of } k_*\, , \nonumber\\
&\leq 2\exp(-2d(k' - k_*+1)) \, , \qquad \textup{since $2^{s} \geq 2s$ for $s=1, 2, 3, \dots$}  
\end{align}

We are now ready to deduce the claimed bound on the probability that a hypercube at level $k=k'+1 \geq k_*+1$ has a child which has seen weakly more demand than its initial supply. Since each (non-leaf) hypercube has $2^d$ children, a union bound tells us that we simply need to multiply the bound in \eqref{eq:prob-running-out-t} by $2^d$, i.e., we obtain a bound of $2\cdot (2/e^2)^{d(k-k_*)} = 2 \nu^{k - k_*}$. 

\end{proof}

\begin{lemma}
Fix $t \in \{0, 1, \dots, N-1\}$ and a leaf hypercube $\hd$. Let $k_{\max}= k_{\max}(t, \hd)$ be the (random) highest level at which the ancestor of $\hd$ has at least one child which has seen (weakly) more demand through period $t$ than its initial supply
$$k_{\max} \triangleq \max \; \{k: 1 \leq k \leq \ell_0, A_k(\hd) \textup{ has a child } h' \textup{ such that } \hat{n}_{h'}(t) \geq n_{h'}(1) \} \cup \{0\}\, .$$
It holds that at no point through period $t$ has a supply or demand unit in hypercube $A_{k_{\max}}(\hd)$ been matched outside of $A_{k_{\max}}(\hd)$. Furthermore, if the demand at period $t+1$ arrives to $\hd$, it will be matched inside $A_{k_{\max}}(\hd)$.
\label{lem:kmax}
\end{lemma}

\begin{proof}
First, note that in the case $k_{\max} = \ell_0$, $A_{k_{\max}}$ is the root node and the lemma holds trivially. Henceforth, assume $k_{\max} < \ell_0 $.

We will use induction on $k$,  
to establish that ``No child of $A_k \triangleq A_k(\hd)$ has had a supply or demand unit matched outside of itself'' for $k \in \{\ell_0, \ell_0-1, \dots, k_{\max}+1\}$.  The root node $A_{\ell_0}$ will form the induction base: Since $k_{\max} < \ell_0 $, we know that the root node has no child which has received weakly more demand through $t$ than its initial supply. And obviously, the root node has never had a supply unit in it get matched to a demand unit outside it (indeed, all demand lives in the root node). Reasoning inductively in time for periods $1, 2, \dots, t$, we obtain that, through $t$, no child of the root has needed to match a demand (supply) in it to supply (demand) outside it, and the claim indeed holds true for the root node. We now proceed to the induction in $k$. Suppose  the claim holds true for $k+1$ and $k \leq k_{\max}+1$. We then know that $A_{k}$, being a child of $A_{k+1}$ has not had a supply or demand unit being matched outside itself through period $t$. With the aforementioned property in place, we are again in the happy situation we encountered at the root node in establishing the induction base, and we deduce as before (using $k > k_{\max}$) that the claim holds true for $k$. Induction completes the proof. 

Using $k=k_{\max}+1$, since $A_{k_{\max}}$ is a child of $A_{k_{\max}+1}$, we deduce the first part of the lemma from the fact established above: at no point through period $t$ has a supply or demand unit in hypercube $A_{k_{\max}}$ been matched outside of $A_{k_{\max}}$. Also, by definition of $k_{\max}$, and considering $A_{k_{\max}+1}$, we know that its child $A_{k_{\max}}$ satisfies $\hat{n}_{A_{k_{\max}}}(t)< n_{A_{k_{\max}}}(1)$. Hence, if the demand at period $t+1$ arrives to $\hd$, there is at least one supply unit remaining in $A_{k_{\max}}$, and the demand will be matched inside $A_{k_{\max}}$.
\end{proof}

Using the previous two lemmas, we now establish a bound on the expected match cost of matching the $(t+1)$-th arrival, which will facilitate the proof of the theorem.
\begin{lemma}
There exists $C \triangleq C(d)< \infty$ such that  the expected cost of matching the $(t+1)$-th demand arrival for $t \in \{0, 1, \dots, N-1\}$ is bounded above by 
$C 2^{\min(k_*, \ell_0)-\ell_0}$,
for $k_*$ defined in Lemma~\ref{lemma:prob-hypercube-in-trouble}.
\label{lem:cost-in-period-tp1}
\end{lemma}

\begin{proof}
    Consider the demand which arrives at period $t+1$, to leaf hypercube $\hd=\hd(t+1)$.

    Note that if $k_* \geq \ell_0$, the claimed bound holds trivially for any $C \geq \sqrt{d}$, since any two points in the unit hypercube are separated by at most $\sqrt{d}$.  Henceforth, suppose $k_*< \ell_0$.

Recall $k_{\max}$ defined in Lemma~\ref{lem:kmax} for period $t$. Note that by definition of $k_{\max}$, we have
\begin{align}
    \Pr (k_{\max} = k) &\leq \Pr (A_{k}(\hd) \textup{ has a child } h' \textup{ such that } \hat{n}_{h'}(t) \geq n_{h'}(1) ) \, \nonumber\\
    &\leq 2 \nu^{k-k_*} \, ,
    \label{eq:kmax_bound}
\end{align}
for $ k \geq \max(k_*+1,1)$, where we used Lemma~\ref{lemma:prob-hypercube-in-trouble} in the second inequality above. 

Lemma~\ref{lem:kmax} assures us that the demand at period $t+1$ will be matched inside $A_{k_{\max}}(\hd)$.
It follows that the match cost in period $t+1$ is no more than the maximum distance between any two points in $A_{k_{\max}}(\hd)$, which is $\sqrt{d}2^{k_{\max}- \ell_0}$. Hence, we have
\begin{align}
    \Ex [\textup{Match cost in period } t+1] &\leq \sum_{k=0}^{\ell_0}  \Pr(k_{\max}=k) \sqrt{d}2^{k- \ell_0} \nonumber \\
    &\stackrel{\textup{(a)}}{\leq} \sqrt{d}2^{k_*-\ell_0} + \sum_{k=k_*+1}^{\ell_0}  \Pr(k_{\max}=k) \sqrt{d}2^{k- \ell_0} \nonumber\\
    &\stackrel{\textup{(b)}}{\leq}  \sqrt{d}2^{k_*-\ell_0} + \sqrt{d} 2^{k_*-\ell_0} \sum_{k=k_*+1}^{\ell_0} 2 \nu^{k-k_*} 2^{k - k_*}\nonumber\\
    &\stackrel{\textup{(c)}}{\leq}  \sqrt{d}2^{k_*-\ell_0}\big (1 + 
    2/(1-2\nu) \big ) \nonumber\\
    &\stackrel{\textup{(f)}}{\leq} 
    (7/2)\sqrt{d}2^{k_*-\ell_0} \, .
    \label{eq:match_cost_bound}
\end{align}
Here, (a) holds from splitting the summation into two parts (recall that $k_*\geq 0$), and upper bounding the first part using $\sqrt{d}2^{k- \ell_0} \leq \sqrt{d}2^{k_*- \ell_0}$ for $k \leq k_*$, and $\sum_{k=0}^{k_*}  \Pr(k_{\max}=k) \leq 1$. (b) follows from using the bound \eqref{eq:kmax_bound} on $\Pr(k_{\max}=k)$. (c) follows from summing the infinite geometric series and using $2 \nu < 1/5$. (f) follows from using $2 \nu < 1/5$. 
Thus we have established the claimed bound follows for $C \triangleq (7/2)\sqrt{d}$. (Note that this value of $C$ also works for the case $k_* \geq \ell_0$ discussed at the beginning of the proof.)
\end{proof}

\begin{proof}[Proof of Theorem~\ref{thm:semi-dynamic}]

The semi-dynamic setting immediately inherits all lower bounds in Theorem~\ref{thm:static-matching} (which considers the static setting), since the semi-dynamic setting is only harder than the static setting due to uncertainty about the future, and the two settings are identical otherwise.

We make use of Lemma~\ref{lem:cost-in-period-tp1} to prove the upper bounds in the theorem for $d \neq 2$, and also for $d=2$ and $M \in [\epsilon N, N]$. We begin with an observation about a key quantity, namely, $2^{k_* - \ell_0}$. From the definition of $k_*$, it follows that there a universal constant $C_1 <\infty$ which does not depend on $t$ and $d$ such that 
\begin{align}
    (1/C_1) \left ( \frac{N}{(M+N-t)^2}\right )^{1/d} \leq 2^{k_* - \ell_0} \leq C_1 \left ( \frac{N}{(M+N-t)^2}\right )^{1/d} \, .
    \label{eq:bound-on-expected-cost}
\end{align}
Now, we know from Lemma~\ref{lem:cost-in-period-tp1} that the expected cost in period $t+1$ is bounded above by $C\min(1,2^{k_* - \ell_0})$ for some $C = C(d) < \infty$.  
Combining with the upper bound in \eqref{eq:bound-on-expected-cost} and summing over $t$ from $0$ to $N-1$, we have that 
\begin{align}
\textup{Total expected cost} \leq C_2 \sum_{t=0}^{N-1}\min \left (1, \left ( \frac{N}{(M+N-t)^2}\right )^{1/d} \right ) \, ,
\label{eq:bound-on-total-cost}
\end{align}
for some $C_2= C_2(d) < \infty$.
What remains is a mechanical exercise of controlling the sum in \eqref{eq:bound-on-total-cost}, whose terms are non-decreasing in $t$, to arrive at the theorem for $d=1$ and $d\geq 3$. We outline this analysis next.  

Consider $d=1$. Let us begin with the case $M \geq \sqrt{N}-1$. Here, $ \frac{N}{(M+N-t)^2}< 1 $ for all $t < N$, and the sum can be bounded above by an integral
\begin{align*}
\sum_{t=0}^{N-1} \frac{N}{(M+N-t)^2}= N\sum_{\tau=M+1}^{M+N} \frac{1}{\tau^2} \leq N \int_M^{M+N} \frac{d\tau}{\tau^2} \, = N\left ( \frac{1}{M} - \frac{1}{M+N} \right) \leq \frac{N}{M} \, .
\end{align*}
(We will repeatedly use the idea of bounding a sum by an integral in analyzing the different cases.)
Since the average expected cost per match is just the total expected cost divided by $N$, we have proved the upper bound of $1/M$ in the theorem for this case. Next consider the case $M < \sqrt{N}-1$. Here, for each $t > M+N -\sqrt{N}$, the corresponding term in the sum is $1$, and the sum is bounded above by
\begin{align*}
    N\cdot \sum_{\tau=\sqrt{N}}^{M+N} \frac{1}{\tau^2} + \sqrt{N} - M \leq N \cdot \frac{1}{\sqrt{N}-1} + \sqrt{N} \leq 3 \sqrt{N} \, .
\end{align*}
Dividing by $N$ to move from sum to average, we have proved the upper bound in the theorem for this case. 

Next consider $d=3$. Note that the expected total cost, and also the upper bound \eqref{eq:bound-on-total-cost}, are non-increasing in $M$. Hence it suffices to prove the upper bound for the case of $M=0$. Here, for each $t > N -\sqrt{N}$, the corresponding term in the sum is $1$, and the sum is bounded above by
\begin{align*}
     N^{1/d}\cdot \sum_{\tau=\sqrt{N}}^{N} \frac{1}{\tau^{2/d}} + \sqrt{N}  \leq N^{1/d} \cdot \frac{N^{1-2/d}}{1-2/d} + \sqrt{N} \leq \left( \frac{2}{1-2/d}\right) N^{1-1/d} \, ,
\end{align*}
where we used $1-1/d > 1/2$ for $d\geq 3$. Dividing by $N$, we obtain the upper bound in the theorem for this case as well. 

For $d=2$ and $M \in [\epsilon N, N]$, the largest cost and the largest bound on the cost occurs for the smallest allowed value of $M$, namely, $M=\epsilon N$. We can bound the sum in \eqref{eq:bound-on-total-cost} as
\begin{align*}
\sum_{t=0}^{N-1} \frac{\sqrt{N}}{M+N-t} = \sqrt{N} \sum_{\tau=\epsilon N+1}^{N(1+\epsilon)} \frac{1}{\tau} \leq \sqrt{N} \int_{\epsilon N}^{N(1+\epsilon)} \frac{d\tau}{\tau} = \sqrt{N} \log ((1+\epsilon)/\epsilon) \, .
\end{align*}
Dividing by $N$ gives us the upper bound in the theorem.

To get our achievability result for the last remaining case of $d=2$ and $M \leq N^{1-\epsilon}$,
we use a matching algorithm different from Hierarchical Greedy. Here we leverage recent work on gravitational matching for uniform points on the surface of a sphere \cite{holden2021gravitational}. 
Using the online (gravitational matching) algorithm in that paper used to prove \cite[Corollary~3]{holden2021gravitational}, along with the translation from matching on the surface of a sphere in $\mathbb{R}^3$ to matching on the unit square via stereographic projection \cite[Proposition 8]{holden2021gravitational}, we obtain a bound of $\tilde{C}_2 \sqrt{\log (N+M-t+1)}/\sqrt{N+M-t+1}$ on the expected cost incurred in period $t$. Summing over $t = 1, 2, \dots, N$, the total expected cost is
\begin{align}
 \tilde{C}_2 \sum_{i=M+1}^{N+M} \sqrt{\frac{\log i}{i}}
 &\leq  \tilde{C}_2 \sqrt{\log (2 N)} \sum_{i=M+1}^{N+M} \frac{1}{\sqrt{i}}
 \leq \tilde{C}_2 \sqrt{2\log N} \int_{M}^{M+N} \frac{1}{\sqrt{x}} dx \\
  &\leq 2 \tilde{C}_2 \sqrt{2N \log N} \, . \nonumber
\end{align}
It follows that the expected cost per match is at most $2 \tilde{C}_2 \sqrt{2\log N}/\sqrt{N}$. This completes the proof of the theorem.
\end{proof}

We conclude this subsection with a brief discussion. Our formal analysis reveals an interesting property, namely, that Hierarchical Greedy incurs a cost at each length scale which is simply ($N$ times) the barrier to matching at that length scale. 
\begin{remark}
Lemma~\ref{lem:cost-in-period-tp1} tells us that, for length scale $L=2^{k_*-\ell_0}$ corresponding to some $0 \leq k_* \leq \ell_0$, for $t$ such that 
\begin{align}
  \sqrt{19\cdot 2^{2d} \cdot N /L^d }=\sqrt{19N 2^{d(\ell_0 - k_* +2)}} &> M+N-t \nonumber \\
  &\geq \sqrt{19N 2^{d(\ell_0 - k_* +1)}} =\sqrt{19\cdot 2^{d} \cdot N /L^d }\, ,
  \label{eq:kstar-condition}
\end{align}
the expected cost is at most $C L$. Hence for\footnote{For smaller $L$, there is no $t$ such that the condition \eqref{eq:kstar-condition} holds.} $L > L_* \sim (N/M^2)^{1/d}$, there is an interval of $t$ of size $\sim \sqrt{N}/L^{d/2}$, such that the expected match cost for those arrivals is (at most) order $L$. Assuming that the match cost for these periods is indeed close to $L$ (not just in expectation), we deduce that the total cost incurred by HG at length scale $L$ is $\sim L \cdot \sqrt{N}/L^{d/2} = \sqrt{N}/L^{d/2-1} \sim N \Delta_L $, i.e., the contribution of length scale $L$ to the expected cost per match scales as the barrier to matching $\Delta_L$ at length scale $L$ (see Section~\ref{subsec:intuition-obstacles}).

To complete this heuristic picture of the cost under HG, we note that Lemmas~\ref{lemma:prob-hypercube-in-trouble} and \ref{lem:kmax} assure us that that for periods $t$ satisfying \eqref{eq:kstar-condition}, the match cost is indeed close to $L$: Lemma~\ref{lem:kmax} tells us that the match distance is at most $2^{k_{\max}-\ell_0} \sqrt{d}$, where $k_{\max}$ is a random variable. In turn, Lemma~\ref{lemma:prob-hypercube-in-trouble} assures us that $k_{\max}$ is typically close to $k_*$ (specifically, it has a subexponential tail above $k_*$ with decay factor larger than 2). Finally, we note that a closer look at the analysis reveals that our upper bounds on match distance are in fact tight up to a constant factor, i.e., our upper bounds correctly capture the scaling of the expected match cost under HG with $N$ and $M$. 
\label{rem:HG-aligns-with-deltaL}
\end{remark}

We note that in the case $d=2$ and $M \leq N^{1-\epsilon}$, Hierarchical Greedy fails to achieve the optimal scaling of cost. The expected cost per match under HG is $\Theta(\log N /\sqrt{N})$; one way to see this is through the lens of Remark~\ref{rem:HG-aligns-with-deltaL}: the barrier to matching at each length scale is same $\Delta_L \sim 1/\sqrt{N}$, and HG incurs cost per match equal to the sum of $\Theta(\log N)$ barriers to matching, each of this size. Gravitational matching \cite{holden2021gravitational} allows to (provably) achieve the optimal scaling  $\Theta (\sqrt{\log N}/\sqrt{N})$. The intuition for HG's inadequacy here is that it fails to take advantage of ``cancellation'' of match distances between different length scales, whereas gravitational matching is able to do so.

\subsection{Analysis of Hierarchical Greedy for the fully dynamic model: Proof of Theorem~\ref{thm:MG-performance}}
\label{subsec:proof-FD-achievability}

The fully dynamic model is substantially more challenging technically than the semi-dynamic model. In order to obtain our achievability bounds using Hierarchical Greedy, we carefully select minimum supply requirements $(\gamma_k)_{k=0}^{\ell_0}$ (see \eqref{eq:gamma-def} below), to trade off between the cost incurred at different levels. Informally speaking, the key tradeoff is that smaller $\gamma_k$ reduces the risk of depletion at level $k$ but increases the risk of depletion at descendants (levels below $k$).

Our definition of the HG algorithm ensures the following key properties:
\begin{itemize}
  \item A demand unit arriving in $h \in \cH_k$ is matched at level $k+1$ or higher (recall the definition \eqref{eq:match_at_level_l}) only if, at that time, $n_h(t) \leq \gamma_k$; this is formalized in Lemma~\ref{lem:no-match-below-gamma} below.
  \item Only supply from ``well-supplied'' regions is used to serve ``outside'' demand from other regions (lines 11--15) which leads to the property formalized in  Lemma~\ref{lem:no-outside-match-below-eta} below.
  \item The number of supply units $n_h(t)$ for $h \in \cH_k$ follows a stochastic process which -- after a short initial transient; see Lemma~\ref{lem:bound-on-Tl} -- stochastically dominates a lazy unbiased random walk which is reflected from below at $\gamma_k$ and reflected from above at $\gamma_{k+1}/2^d  \geq \gamma_k$; see Lemma~\ref{lem:Wh-is-a-LB} below. As a result the probability of being at the lower boundary $\gamma_k$ is at most $1/\min(\gamma_{k+1}/2^d - \gamma_k,1)$ in steady state (and not much more than that after the mixing time has elapsed); see Lemma~\ref{lem:RW-time-at-lower-boundary}. Moreover, the fraction of demand that gets matched at level $k+1$ is at most this probability (by the first bullet) and hence inherits the same upper bound; see Lemma~\ref{lem:cost-at-level-ell}.
\end{itemize}

\begin{lemma}[Lower boundaries are respected]
For any level $k \leq \ell_0-1$ and any hypercube $h \in \cH_k$, if $n_h(t) \leq \gamma_k$ then Hierarchical Greedy does not use a supply unit from $h$ for matching with the demand unit in period $t$.
\label{lem:no-match-below-gamma}
\end{lemma}
\begin{proof}
Suppose $h \in \cH_k$ and $n_h(t) \leq \gamma_k$. Since $n_h(t) \leq \gamma_k$, by definition of $\ell$ in lines 5-10 of HG we cannot have both $\ell = k$ AND $A_{\hd}(\cdot)=h$, i.e., a demand unit clearly cannot be matched at level $k$ to hypercube $h$. So the only way it can be matched in hypercube $h$ is that it is matched at level $\ell > k$ in $h' = A_{\ell}(\cdot)$ in
line 11 which is an ancestor of $h$, and that $h$ is subsequently selected as the child of choice on
line 13 enroute to a leaf. But we show that even this is not possible.
By definition of $\ell$, we have $n_{h'} > \gamma_\ell$. We argue by induction that this holds for every node in the path traversed by lines 12-14 to a leaf:
If a node at level $k'+1$ has supply (strongly) exceeding $\gamma_{k'+1}$ its best-supplied child (of its $2^d$ children) has supply (strictly) exceeding 
$\gamma_{k'+1}/2^d \geq \gamma_{k'}$ by \eqref{eq:gamma_k-LB}. Hence $h$ is not on the path traversed, i.e., a supply unit from $h$ is not used for matching.
\end{proof} 

\begin{lemma}[Only well-supplied regions are used to serve demand from outside]
\label{lem:no-outside-match-below-eta}
For any level $k \leq \ell_0-1$, any hypercube $h \in \cH_k$ and any period $t$, if $n_h(t) < \lceil \gamma_{k+1}2^{-d} \rceil$, 
and if the demand unit in period $t$ is not located in $h$, then a supply unit from $h$ is {not} used to serve it.
\end{lemma}
\begin{proof}
If the demand from outside $h$ is matched to a supply unit in $h$, by definition of HG (lines 11--15) it is  matched at a level $\ell$ strictly exceeding $k$. As in the proof of Lemma~\ref{lem:no-match-below-gamma}, it then follows that if the supply unit is taken from $h$, there is more than $\gamma_{k+1}/2^d$ supply in $h$.
\end{proof}

\begin{lemma}[All lower boundaries are reached soon and remain satisfied thereafter]
\label{lem:bound-on-Tl}
Assume $\ell_0$ satisfies $2^{d\ell_0} \leq m$. For any $k \leq \ell_0$ and any $h \in \cH_k$, under the HG algorithm we have $n_h(t) \geq \lfloor \gamma_k \rfloor$ for all $t \geq T_h$ w.p. 1, where $T_h \triangleq \inf\{t \geq 1: n_h(t) \geq \lfloor \gamma_k \rfloor\}$. Moreover, $\Ex[T]\leq  7m (\log m+1) = O(m\log m)$ for $T \triangleq \max_{h \in \cH_0 \cup \cH_1 \cup \dots \cup \cH_{\ell_0}} T_h$.
\end{lemma}
\begin{proof}
Permanence of satisfying lower bounds follows from Lemma~\ref{lem:no-match-below-gamma}: any undersupplied hypercube only experiences supply arrivals and no supply departures and hence $n_h(t) > \gamma_k -1$ continues to hold for all $t \geq T_h$.

We now bound the duration of the transient $T_h$ for an individual hypercube $h \in \cH_k$. For all $t< T_h$, there is no supply departure from $h$ and there is a supply arrival to $h$ with probability $p=2^{d(k - \ell_0)}$. Consider $\tau \triangleq m(C_\tau \log m +1)$. We have
\begin{align*}
  \Pr (T_h > \tau) &= \Pr \left (\textup{Binomial}(\tau, 2^{d(k - \ell_0)}) \leq \gamma_k -1 \right )\\
  &\leq \Pr \left (\textup{Binomial}(\tau, 2^{d(k - \ell_0})) \leq m 2^{d(k - \ell_0)} \right )\\
  &\leq \exp \left (- \frac{(1/2)(Cmp \log m)^2}{(4/3)Cmp \log m} \right ) = \exp \left ( - \frac{3Cmp \log m}{8} \right ) \,\\
  &\leq \exp \left ( - (3/8)C_\tau \log m \right ) = 1/m^{3C_\tau/8}
\end{align*}
where the second inequality follows from Bernstein's inequality along with the observation that $p = 2^{d(k - \ell_0)} \geq 2^{-d\ell_0} \geq 1/m$; the latter observation is used again in the third inequality. 
The number of hypercubes  $h$ in $\cH_0 \cup \cH_1 \cup \dots \cup \cH_{\ell_0}$ is bounded above by $2^{\ell_0 d}(1 + 2^{-d} + 2^{-2d} + \dots) \leq 2m$ using $2^{\ell_0 d} \leq m$. Hence, it follows from a union bound that for $T \triangleq \max_{h \in \cH_0 \cup \cH_1 \cup \dots \cup \cH_{\ell_0}} T_h$
\begin{align*}
  \Pr (T > \tau) &\leq \sum_h \Pr (T_h > \tau) \leq 2m/m^{3C_\tau/8} \, .
\end{align*}
Let $\tau_0 \triangleq m(3 \log m +1)$; the idea is that this union bound is useful for $C_\tau> 3$, i.e., for  $\tau > \tau_0$. We deduce the claimed bound on $\Ex[T]$ as follows:
\begin{align*}
  \Ex [T] &= \int_{\tau=0}^\infty \Pr (T > \tau) d\tau\\
  &\leq \tau_0 + \int_{\tau=\tau_0}^\infty \Pr (T > \tau)d\tau \leq m (3 \log m + 7) \, ,
\end{align*}
where we used that $\int_{\tau=\tau_0}^\infty \Pr (T > \tau)d\tau \leq 6 m$ which follows from the union bound above and changing the variable of integration to $C_\tau$.
\end{proof}

For each $h \in \cH_k\,,\; k \in \{0, 1, \dots, \ell_0-1\}$, we define the lazy, doubly reflected random walk $W_h$ as follows for $t \geq T_h$: Define $\lrb  \triangleq  \lfloor \gamma_k  \rfloor$ and $\urb \triangleq \lceil \gamma_{k+1}/2^d \rceil -1$. These will be the lower and upper reflecting boundaries, respectively. The random walk $W_h$ starts at\footnote{Note that, by definition of $T_h$, we have $W_h(T_h) \geq \lrb$.} $W_h(T_h) = \min(n_h(T_h), \urb)$ and thereafter evolves as per
\begin{align}
  \tilde{W}_h(t) = \max\!\big (\, \lrb\, , \, W_h(t) - \mathbb{I}\big [h = A_k(\hd(t))\big ] \, \big )\nonumber\\
  W_h(t+1) = \min\!\big (\, \tilde{W}_h(t) + \mathbb{I}\big [h = A_k(\hs(t)) \big ]\, , \,\urb \,\big)
\label{eq:W-def}
\end{align}
for $t \geq T_h$, where $\hd(t)$ and $\hs(t)$ are the leaf hypercubes where demand and supply arrive in period $t$.
The idea is that $n_h(t)$ is bounded below by $W_h(t)$, and that $\tilde{n}_h(t)$  is bounded below by $\tilde{W}_h(t)$, where $\tilde{n}_h(t)$ is defined as the number of supply units in $h$  after the demand at period $t$ arrives but before the supply at period $t$ arrives.

\begin{lemma}[$W_h$ is a lower bound for supply $n_h$]
For all $t \geq T_h$, we have $n_h(t) \geq W_h(t)$ and $\tilde{n}_h(t) \geq \tilde{W}_h(t)$.
\label{lem:Wh-is-a-LB}
\end{lemma}
\begin{proof}
We prove the lemma by induction on $t$.

We begin by establishing the induction base for $t= T_h$. This is straightforward; by definition of $W_h(T_h)$, we know that $n_h(T_h) \geq  W_h(T_h)$. And from Lemma~\ref{lem:no-match-below-gamma} (Hierarchical Greedy respects the lower boundary) and the definition of $\tilde{W}_h(t)$ for $t = T_h$, we know that $\tilde{n}_h(T_h) \geq \tilde{W}_h(T_h)$.

We will now show that if, for some $t$, we have $n_h(t) \geq W_h(t)$, then $\tilde{n}_h(t) \geq \tilde{W}_h(t)$ and that this further implies $n_h(t+1) \geq W_h(t+1)$. Induction will then complete the proof.
\smallskip

{$\mathbf{n_h(t) \geq W_h(t) \Rightarrow \tilde{n}_h(t) \geq \tilde{W}_h(t)}$.}
Suppose $n_h(t) \geq W_h(t)$. If the inequality is strict then $n_h(t) \geq W_h(t)+1$ and $\tilde{n}_h(t) \geq \tilde{W}_h(t)$ follows immediately since $\tilde{n}_h(t) \geq n_h(t) -1$ by definition of $\tilde{n}_h(t)$ (at most one supply unit from $h$ is used to serve the newly arrived demand), and $W_h(t) \geq \tilde{W}_h(t)$ by definition of $\tilde{W}_h(t)$. It remains to consider the case $n_h(t) = W_h(t)$. We consider two subcases corresponding to whether or not the demand arrival at period $t$ is located in $h$, and show the desired result in each case.
\begin{itemize}
  \item Suppose $h$ is an ancestor of $\hd(t)$, i.e., the demand arrival at period $t$ is located in $h$. Then, if $W_h(t)$ is not at the lower reflecting boundary $W_h(t) \geq \lrb + 1$, then by definition $\tilde{W}_h(t) = W_h(t) -1$, and we immediately infer that $\tilde{n}_h(t) \geq n_h(t) -1 \geq W_h(t) -1=\tilde{W}_h(t)$. On the other hand, if $n_h(t)=W_h(t) = \lrb$, then by definition of $\lrb$ and $\cS$, $h$ is undersupplied $k \in \cS$, and the HG algorithm uses a supply node from outside $h$ to serve the newly arrived demand, i.e., $\tilde{n}_h(t) = n_h(t) = W_h(t) = \tilde{W}_h(t)$.
  \item The other possibility is that the demand arrival at period $t$ is not located in $h$. Since $n_h(t) = W_h(t) \leq \urb < \gamma_{k+1}2^{-d}$, by Lemma~\ref{lem:no-outside-match-below-eta} we know that HG does not use supply from $h$ to serve the newly arrived demand. As a result $\tilde{n}_h(t) = n_h(t) \geq W_h(t) = \tilde{W}_h(t)$.
\end{itemize}

\smallskip
{$\mathbf{\tilde{n}_h(t) \geq \tilde{W}_h(t) \Rightarrow n_h(t+1) \geq W_h(t+1)}$.}
If the new supply arrival at period $t$ is located in $h$ then $n_h(t+1) = \tilde{n}_h(t)+1 \geq \tilde{W}_h(t) +1 \geq W_h(t+1)$ by definition of $W_h(t+1)$, which is the result we seek. On the other hand, if the new supply arrival at period $t$ is not located in $h$, then $n_h(t+1) = \tilde{n}_h(t) \geq \tilde{W}_h(t) = W_h(t+1)$, which is again the result we seek.

\smallskip
Induction completes the proof.

\end{proof}

\begin{lemma}[Probability that $W_h$ is at its lower boundary]
There exists a universal constant $C< \infty$, such that for any level $k \leq \ell_0-1$, any hypercube $h \in \cH_k$ and for all $t \geq \hat{T}_h  \triangleq T_h + C 2^{(\ell_0-k)d}(\urb-\lrb+1)^2 \log m$, we have
\begin{align*}
  \Pr (W_h(t) = \lrb) \leq 2 /(\urb - \lrb + 1) \, .
\end{align*}
In the case that the initial supply is evenly spread out,  for all $t\geq 1$ we have $\Pr (W_h(t) = \lrb) \leq 1/(\urb - \lrb + 1)$.
\label{lem:RW-time-at-lower-boundary}
\end{lemma}
\begin{proof}
The offset by $T_h$ is simply because $W_h$ starts at period $T_h$.
By definition, in each period $t \geq T_h$, $W_h$ first takes an downward step with probability $q_k\triangleq 2^{-(\ell_0-k)d}$ (unless it is at the lower boundary $\lrb$) and then takes a upward step with probability $q_k$ (unless it is at the upper boundary $\urb$). In words, $W_h(t)$ is a lazy doubly reflected simple random walk with a probability $p_k = 2 q_k (1-q_k)$ of taking a step (in a given period) if it is not at one of the boundaries, a probability $q_k$ of taking an upward step if it is at the lower boundary $W_h(t)$, and a probability $q_k(1-q_k)$ of taking a downward step if it is at the upper boundary. 

Let $\pi$ be the stationary distribution of $W_h$ with $\pi(k)$ being the probability of $W_h(t) = k$ in steady state. 
Then, $\pi$ satisfies detailed balance with $\pi^* \triangleq \pi({\lrb+1})= \pi({\lrb+2})= \dots = \pi({\urb-1}) = \pi({\urb})$ 
and $\pi_{\lrb} q_k = \pi_{\lrb+1} q_k (1-q_k)$, i.e., $\pi_\lrb = \pi^* (1-q_k)$. In words, $\pi$ is a uniform distribution over the possible values of $W_h$, except at the lower boundary $\lrb$, whose steady state probability of occurrence is slightly smaller than at the other possible values. It follows that $\pi({\lrb}) < 1/(\urb -\lrb +1)$. 

Standard arguments imply that the mixing time of  $W_h(t)$ is no more than $1+2(\urb-\lrb)^2/q_k$. We describe an argument based on \cite{levin2017markov} here:  Define the mixing time $\Tmix$ of a Markov chain as per Eq.~(4.33) in \cite{levin2017markov}, i.e., the smallest number of periods such that, from any starting state, the total variation distance between the distribution of the final state and the stationary distribution is no more than 1/4. Define the worst-case hitting time $\Thit$ as per Eq.~(10.5) in \cite{levin2017markov}, i.e., the worst case over starting state $\omega_1$ and final state $\omega_2$, of the expected time it takes to first hit state $\omega_2$ starting from state $\omega_1$. Note that $W_h(t)$ is a reversible Markov chain. Hence we can use its worst-case hitting time to bound its mixing time as per inequality~(10.24) in \cite{levin2017markov}:
\begin{align}
    \Tmix \leq 2 \Thit + 1 \, .
    \label{eq:Tmix-bounded-Thit}
\end{align}
A straightforward argument allows us to compute $\Thit$. First, note that the worst-case hitting time corresponds to the expected time to hit $\lrb$ when starting from $\urb$. (Recall the asymmetry in the behavior of the two boundaries; in particular, 
the walk is lazier at $\urb$ than at $\lrb$.) 
Now, we can generate $W_h(t)$ by considering a lazy random walk on a cycle with nodes labeled as $-(\urb - \lrb)+1, -(\urb - \lrb)+2, \dots, 0, \dots,  (\urb - \lrb)$, with the interpretation that node $ -(\urb - \lrb)+1 \leq i \leq (\urb - \lrb)-1$ on the cycle maps to $W_h = \urb - |i|$. 
At each node on the cycle, the random walk has equal likelihood of taking a step to the left or to the right (to close the cycle, the node $\urb - \lrb$ is treated as being immediately to the left of node $-(\urb - \lrb)+1$). The probability of taking a step at each node is $2q_k (1-q_k)$, with the exception of node $(\urb - \lrb)$, where the probability of taking a step is $2 q_k$. $\Thit$ is simply the expected hitting time of node $\urb - \lrb$ starting from node $0$, for the random walk on the cycle. And this time, using \cite[Proposition 2.1,][]{levin2017markov} regarding Gambler's ruin, is simply $\Thit = (\urb - \lrb)^2/(2 q_k (1-q_k)) \leq (\urb - \lrb)^2/q_k$. It follows from \eqref{eq:Tmix-bounded-Thit} that 
\begin{align}
    \Tmix \leq 1+2(\urb - \lrb)^2 /q_k \leq 2(\urb - \lrb+1)^2 /q_k =2(\urb - \lrb+1)^2  2^{(\ell_0-k)d}\, .
\end{align}

From the definition of $\Tmix$, and the fact that $W_h$ starts at time $T_h$, we know that the distribution of $W_h(t)$ for all $t\geq T_h + \lceil \log m \rceil \Tmix$ has total variation distance from the stationary distribution $\pi$ of at most $(1/4)^{\log m} \leq 1/m \leq 1/(\urb-\lrb+1)$, where the last inequality used $\urb= \lceil \gamma_{k + 1}2^{-d} \rceil -1 \leq \lceil m2^{-d} \rceil -1 \leq m-1$. Now $\lceil \log m \rceil \Tmix \leq 2 \log m \cdot 2  (\urb - \lrb+1)^2  2^{(\ell_0-k)d} = 4 \cdot 2^{(\ell_0-k)d}  (\urb - \lrb+1)^2 \log m$.
Combining with $\pi({\lrb}) < 1/(\urb -\lrb +1)$, we get that for all $t \geq \hat{T}_h \triangleq T_h + 4 \cdot 2^{(\ell_0-k)d}  (\urb - \lrb+1)^2 \log m$, we have
\begin{align*}
  \Pr (W_h(t) = \lrb) &\leq \pi(\lrb) +  \textup{(TV distance of $W_h(t)$ from $\pi$)} \\
  &\leq 2/(\urb-\lrb+1) \, .
\end{align*}
This completes the proof of the first part of the lemma.


In the case that the initial supply is evenly spread out,  $W_h$ starts at $W_h(1) = \urb$ and $T_h=1$. In this case, the distribution of $W_h(t)$ stochastically dominates $\pi$ for all $t \geq 1$ (this is immediate using induction on $t$), and hence \begin{align*}
  \Pr (W_h(t) = \lrb) &\leq \pi(\lrb)< 1/(\urb-\lrb+1) \, 
\end{align*}
for all $t \geq 1$.
\end{proof}

\begin{lemma}[Probability of being matched at level $k$]
For any level $\ell \in \{1, 2, \dots, \ell_0\}$, any hypercube $h \in \cH_\ell$ and in any period $t \geq \hat{T}_h$, the probability of a demand unit being matched at level $\ell$ is bounded above by $2 /(\overline{R}_{\ell-1}-\underline{R}_{\ell-1}+1)$.
Hence, for $t \geq \hat{T} \triangleq \max_{h \in \cH_0 \cup \cH_1 \cup \dots \cap \cH_{\ell_0-1}} \hat{T}_h$, the expected cost per match due to matches at level $\ell$ is bounded above by $C'2^{\ell-\ell_0} /(\overline{R}_{\ell-1}-\underline{R}_{\ell-1}+1)$ for some $C' = C'(d) < \infty$ which does not depend on $m$ and $\ell$. (In the case that the initial supply is evenly spread out, this holds for any period $t \geq 1$.)
\label{lem:cost-at-level-ell}
\end{lemma}
\begin{proof}
For a match to occur at level  $\ell \geq 1$, it must be that the level $\ell-1$ ancestor $h'=A_{\ell-1}(h)$ of the arrival leaf node $h$ is in $\cS$, i.e., is undersupplied $n_{h'}(t) \leq \gamma_{\ell-1}$, by definition of $\ell$ in lines 5--10 of HG. But for $t \geq \hat{T}_{h'}> T_{h'}$, this requires $W_{h'}(t) = n_{h'}(t) = \underline{R}_{\ell-1}$ by Lemma~\ref{lem:Wh-is-a-LB}, which occurs with probability at most $2/(\overline{R}_{\ell-1}-\underline{R}_{\ell-1}+1)$, by Lemma~\ref{lem:RW-time-at-lower-boundary}. (In the case that the initial supply is evenly spread out, this holds for any period $t \geq 1$.)

Now, for each level $\ell$ hypercube has side length $2^{\ell-\ell_0}$, and hence the maximum distance between any pair of points in the same level-$\ell$ hypercube is bounded  by $(C'/2) 2^{\ell-\ell_0}$ for some $C' = C'(d) < \infty$ which does not depend on $m$ and $\ell$. It follows that the matching distance for a match at level $\ell$ is bounded above by $(C'/2) 2^{\ell-\ell_0}$. Multiplying with the upper bound $2/(\overline{R}_{\ell-1}-\underline{R}_{\ell-1}+1)$ on the probability of matching at level $\ell$ we get the expected cost per match due to matches at level $\ell$.
\end{proof}

Using Lemma~\ref{lem:cost-at-level-ell}, we know that the expected cost per match at level $\ell \geq 1$ is bounded above by $C'2^{\ell-\ell_0}/(\overline{R}_{\ell-1}-\underline{R}_{\ell-1}+1) \leq  C'2^{\ell-\ell_0}/(\gamma_{\ell}2^{-d} - \gamma_{\ell-1})$, for all $t \geq \hat{T}$ and for all $t\geq 1$ if the supply is initially evenly distributed. For the same $C'$ the maximum distance between any pair of points which lie within the same level $0$ hypercube is bounded above by $(C'/2)2^{-\ell_0}$, and hence this is an upper bound on the cost per match at level $0$. Summing the cost per match at level $\ell$ over $\ell =0, 1, \dots, \ell_0$, the expected cost per match for $t \geq \hat{T}$ is bounded above by
\begin{align}
 \Ex[\textup{Distance between matched pair at }t \, |\, t \geq \hat{T}] &\leq (C'/2)2^{-\ell_0} + C' \cdot 2^{-\ell_0} \sum_{\ell = 1}^{\ell_0} 2^\ell/ (\gamma_{\ell}2^{-d} - \gamma_{\ell-1}) \, .
 \label{eq:expected-match-cost-general}
\end{align}

Our approach for establishing Theorem~\ref{thm:MG-performance} will be to choose the $(\gamma_k)$s in a manner that they satisfy $\gamma_{k}2^{-d} - \gamma_{k-1} = \beta^k$ for $k = 1, 2, \dots, \ell_0$. Specifically, we will define
\begin{align}
\gamma_{k} &\triangleq m2^{-(\ell_0-k)d}- \sum_{k' = k}^{\ell_0}
\beta^{k'}2^{-d(k'-k)} \, , \qquad \textup{for } k = 0, 1, \dots, \ell_0 \, .
\label{eq:gamma-def}
\end{align}
The idea behind our definition of $\gamma_k$ is to introduce an additional slack of $\beta^k$ as we go down from level $k+1$ to level $k$ of the tree, for each $k$. Note that $\gamma_{k+1}2^{-d} = \gamma_k + \beta^k > \gamma_k$ for all $k \leq \ell_0-1$, as required,
and $\gamma_{\ell_0} < m$ holds by definition. With this choice of $(\gamma_k)$s, we have, using \eqref{eq:gamma-def}, that
\begin{align}
 \Ex[\textup{Distance between matched pair at }t \, |\, t \geq \hat{T}] &\leq (C'/2)2^{-\ell_0} + 2C' \cdot 2^{-\ell_0} \sum_{\ell = 1}^{\ell_0} (2/\beta)^{\ell-1} \, .
 \label{eq:expected-match-cost-beta}
\end{align}
We will make use of $\beta = 2.01$ for $d\geq 2$ and $\beta = 2$ for $d=1$ to establish the theorem. In particular, these choices of $\beta$ will ensure $\gamma_k \geq 0$ for all $k$. 

\begin{proof}[Proof of Theorem~\ref{thm:MG-performance}]
{\bf Proof for $d \geq 2$.} We set $\ell_0$ to be the largest integer with $2^{d\ell_0} \leq m/4$.
This choice of $\ell_0$ ensures that there are $\Theta(m)$ leaf hypercubes, and the average number of supply units per leaf hypercube is $\Theta(1)$ (in particular, it is at least $4$ units). 
We fix $\beta = 2+ \delta \in (2, 3)$ where $\delta \in (0,1)$ will be chosen later, and $\textup{for } k = 0, 1, \dots, \ell_0$, set $\gamma_k$ as per \eqref{eq:gamma-def}.
Note that
\begin{align*}
\gamma_k \geq   4\cdot 2^{k d} -\beta^k \cdot \frac{1}{1-\beta2^{-d}}
 \geq 4\cdot 2^{k d} - 4\cdot 3^k \geq 0 \,
\end{align*}
as required, where the first inequality used the definition of $\ell_0$, the second inequality used $\beta < 3$ and $\frac{1}{1-\beta2^{-d}} < \frac{2^{d}}{2^{d}-3} \leq 4$ for all $d\geq 2$, and the last inequality used $d \geq 2$.

Using \eqref{eq:expected-match-cost-beta}, the expected cost per match for $t \geq \hat{T}$ is bounded above by
\begin{align}
 \Ex[\textup{Distance between matched pair at }t \, |\, t \geq \hat{T}] &\leq (C'/2)2^{-\ell_0} + 2 C' \cdot 2^{-\ell_0} \sum_{\ell = 1}^{\ell_0} (2/\beta)^{\ell-1}\nonumber\\
 &\leq \Big(1/2+\frac{2}{1-2/\beta}\Big ) C' 2^{-\ell_0} \nonumber\\
 &\leq (3/2+2/\delta)C' 4m^{-1/d} \leq (14/\delta) C' m^{-1/d}  \, ,
 \label{eq:cost-after-That}
\end{align}
where the second inequality follows from plugging in $\beta = 2+\delta$ and the definition of $\ell_0$ which implies $2^{d(\ell_0+1)} > m/4 \Rightarrow 2^{-\ell_0}< 2\cdot 4^{1/d }\cdot m^{-1/d} \leq 4 m^{-1/d}$ for $d \geq 2$. This completes the proof for the case where the supply is initially evenly distributed.

To accommodate arbitrary initial locations of supply units, we now show that $\Ex[\hat{T}]= o(m^{1.01})$ for an appropriate  choice of $\delta$. Now
 \begin{align*}
   \Ex[\hat{T}] &\leq \Ex[T] + \max_{k \leq \ell_0-1} C 2^{(\ell_0-k)d}(\urb-\lrb+1)^2 \log m  \\
   &\leq 7m \log m + C (m/4) \log m \max_{k\leq \ell_0-1} (\beta^k+2)^2 2^{-2l}\, ,
   \end{align*}
 where the second inequality uses Lemma~\ref{lem:bound-on-Tl}, $2^{\ell_0 d} \leq m/4$, $\gamma_{k+1}2^{-d} - \gamma_k = \beta^k$ along with the definitions of $\urb$ and $\lrb$, and $d \geq 2 \Rightarrow 2^{-dk}\leq 2^{-2k}$.
   We have $\beta^k+2 \leq 2\beta^k$ for $k \geq 1$ and the maximizer of $4\beta^{2k}2^{-2k}$ is at $k = \ell_0 -1$, leading to
 \begin{align*}
 \max_{k\leq \ell_0-1} (\beta^k+2)^2 2^{-2l}
   \leq  4 (\beta/2)^{2(\ell_0-1)} = 4 (1+\delta/2)^{2(\ell_0-1)} \leq 4 m^{\delta/(2 \log 2)} \, ,
 \end{align*}

Plugging back into the bound on $\Ex[\hat{T}]$ we get
 \begin{align}
   \Ex[\hat{T}] &\leq (C+7)m^{1+\delta/(2 \log 2)} \log m \, = o(m^{1.01}) \, ,
   \label{eq:That-ub}
\end{align}
by choosing $\delta \triangleq 0.01$.
The total cost which accrues prior to $\hat{T}$ is bounded above by $\hat{T}$ since the cost per match is at most $1$, and hence the contribution of the cost of this initial transient to the average cost over $N$ periods is bounded above by $\Ex[\hat{T}]/N$. To get our result we will require that $N$ is large enough that this contribution is no more than $(C'/\delta)m^{-1/d}$. Specifically, we define $N_d (m)\triangleq (C+7)(\log m) m^{1+\delta/(2 \log 2)+1/d}  (\delta/C') = O(m^{1.01+1/d})$ and demand $N \geq N_d$ so that, using \eqref{eq:cost-after-That} and \eqref{eq:That-ub}, the expected average cost per match is bounded above by
\begin{align*}
  \frac{\Ex[\hat{T}]}{N} + \frac{(N-\hat{T})_+}{N}(14C'/\delta)m^{-1/d} &\leq \Ex[\hat{T}]/N_0 + (14C'/\delta)m^{-1/d}\\
  &\leq (C'/\delta)m^{-1/d} + (14C'/\delta)m^{-1/d} \leq (15C'/\delta)m^{-1/d}
\end{align*}

 We remark that this argument provides the desired bound on expected cost per match for any fixed $\beta \in \big(\, 2\,,\, (3/4)\cdot 2^d\, \big)$, under a long enough horizon (e.g., in steady state). Here $\beta < (3/4)\cdot 2^d$ is needed to ensure that the $\gamma_k$s are non-negative (the contribution of slack from higher levels of the tree decays exponentially), whereas $\beta> 2$ ensures that the contribution to cost from level $\ell$ of the tree decays exponentially in $\ell$. Our choice of $\beta = 2.01$, just above the minimum requirement $\beta>2$, ensures that the transient $\hat{T}$ is short by ensuring that the auxiliary random walks at the higher levels of the hierarchy mix rapidly (though a more refined analysis would control the cost during the transient and obtain a similar bound even for the case where a larger $\beta$ is used).

\smallskip
{\bf Proof for $d =1$.}
For $d=1$, we set $\ell_0$ to be the largest integer with $2^{\ell_0} \leq m/(1+\log_2 m)$. This choice of $\ell_0$ ensures that there are $\Theta(m/\log m)$ leaf hypercubes, and the average number of supply units per leaf hypercube is $\Theta(\log m)$ (in particular, it is at least $1+\log_2 m$). We fix $\beta = 2$ and use the definition \eqref{eq:gamma-def} for the $\gamma_k$s. Plugging in $\beta=2$ and $d=1$ we have
$$\gamma_k = m 2^{-(\ell_0-k)} - 2^k(\ell_0 - k+1) \geq 2^k (1+\log_2 m  - \ell_0 +k -1 )\geq 0 \, ,$$
using the definition of $\ell_0$ which implies $m 2^{-\ell_0} \geq 1+\log_2 m$. 

Using \eqref{eq:expected-match-cost-beta}, the expected cost per match for $t \geq \hat{T}$ is bounded above by $(C'/2)2^{-\ell_0} + 2C' \cdot 2^{-\ell_0} \ell_0 \leq 2C' (1+\ell_0)2^{-\ell_0} \leq 2C' \cdot (1+ \log_2 m) \cdot 2 (1+ \log_2 m)/m=4C' (1+ \log_2 m)^2/m$, where we used that the definition of $\ell_0$ implies $2^{-\ell_0} \leq 2 (1+ \log_2 m)/m$ and $\ell_0 < \log_2 m$. This completes the proof for the case where the supply is initially evenly distributed. Note that here, the bound on the cost due to matches at level $\ell$ is independent of $\ell$ for our choice of $\beta =2$; each term in the sum in \eqref{eq:expected-match-cost-beta} is the same.

Towards accommodating arbitrary initial locations of supply units, we now show that $\hat{T} = O(m^{2})$. We have 
\begin{align*}
 \Ex[\hat{T}] &\leq \Ex[T] + \max_{k \leq \ell_0-1} C 2^{\ell_0-k}(\urb-\lrb+1)^2 \log_2 m  \leq 7m \log m + C m \max_{k\leq \ell_0-1} (2^k+2)^2 2^{-k}\, ,
\end{align*}
 where the second inequality uses Lemma~\ref{lem:bound-on-Tl}, $2^{\ell_0 d} \leq m/(1+\log_2 m) \Rightarrow 2^{\ell_0 d} \log_2 m \leq m$, and $\gamma_{k+1}2^{-d} - \gamma_k = 2^k$ along with the definitions of $\urb$ and $\lrb$. Now
 \begin{align*}
   \max_{k\leq \ell_0-1} (2^k+2)^2 2^{-k} &\leq \max( (2^{\ell_0})^2 2^{-\ell_0+1}, 9) = \max( 2^{\ell_0+1}, 9) \\
   &\leq \max( 2 m /(1+ \log_2 m), 9) \leq 5 m \, ,
 \end{align*}
 where we used that for $\ell_0 \geq 3$, the maximizer occurs at $k = \ell_0 -1$ and $2^{\ell_0-1} + 2 \leq 2^{\ell_0}$, and $m \geq 2$ in the last inequality. Plugging back into the previous bound, we obtain
\begin{align*}
 \Ex[\hat{T}] &\leq 7m \log m +  5 Cm^2    \leq (5C+7)m^2 =O(m^2)\, ,
 \end{align*}
 using $\log m \leq m$. We define $N_0 (m)\triangleq (5C+7)m^3$ and require $N \geq N_0$ so that, as in the proof for $d\geq 2$, the expected average cost per match is bounded above by
\begin{align*}
  \frac{\Ex[\hat{T}]}{N} + \frac{(N-\hat{T})_+}{N} \cdot 4C' (1+ \log_2 m)^2/m &\leq \Ex[\hat{T}]/N_0 + 4C' (1+ \log_2 m)^2/m\\
  &\leq m^{-1} + 4C' (1+ \log_2 m)^2/m \\
  &\leq (4C'+1)(1+ \log_2 m)^2/m \, .
\end{align*}
This completes the proof.

We remark that in this case no other choice of $\beta$ works. Our choice $\beta = 2$ balances between the challenge of ensuring non-negativity of the $\gamma_k$s and ensuring that the cost at level $\ell$ does not grow with $\ell$. Our upper bound in this case is a factor $(\log m)^2$ larger than the nearest-neighbor-distance $\Theta (1/m)$. Since all levels of the tree contribute the same order of cost, this accounts for one factor $\log m$ in our upper bound. The second factor of $\log m$ enters since the slack introduced at all levels of the tree in the definition of the $\gamma_k$s adds up as $\ell_0$ equal terms (in our definition of $\gamma_0$), we need to have $\Theta(\log m)$ average supply at each leaf node to obtain $\gamma_0 \geq 0$ instead of $\Theta(1)$ average supply per leaf node needed for $d \geq 2$.
\end{proof}

\section{Application: Capacity planning for shared mobility systems}
\label{sec:capacity-planning}

Motivated by applications like ride-hailing and bikesharing, consider the fully dynamic model, augmented (interpreted) as follows:
\begin{itemize}
  \item We define physical time as $t/n$ where $n$ will be our scaling parameter, interpreted as the rate of arrival of demand with respect to physical time. We call $n$ the \emph{load factor}.
  \item Our system will have $n + m$ supply units circulating in it. New supply units will not arrive and existing supply units will not leave.
  \item Each supply unit becomes ``busy'' when it is matched to a demand unit. It remains busy for $n$ periods (a physical time interval of length 1, the time taken to complete a ride) and then reappears as a free supply unit at a uniformly random location. With regard to the ride-hailing application, the interpretation is that each demand unit has an i.i.d. uniformly random origin and an i.i.d. uniformly random destination. In each period, one supply unit becomes free (``arrives'') at a uniformly random location.\footnote{This includes the initial periods $1, 2, \dots, n$. Notionally, those supply units were already in transit at period 1. }
  \item The platform is able to choose the excess supply $m \geq 0$ beforehand. ($m$ cannot vary over time.) Excess supply costs $m$ per unit of physical time, i.e., $m/n$ per period.
  \item As before, the matching cost is equal to the distance between the demand unit and the matched supply unit.
  \item The platform's objective is to choose the excess supply $m$ (and the matching policy) so as to minimize the sum of the excess supply cost and the expected matching cost, over some time horizon.
  \item For simplicity, assume that initially the $m$ free supply units are at evenly spread locations.
\end{itemize}

Consider $d\geq 2$. For any fixed $m$, Theorem~\ref{thm:MG-performance} and Proposition~\ref{prop:closest-nhbr-lower-bound} tell us that the smallest achievable expected cost per period is $\Theta(1/m^{1/d})$, over any time horizon. The supply cost per period is $m/n$. Hence the total expected cost per period is $m/n + \Theta(1/m^{1/d})$ which is minimized by choosing $m = \Theta(n^{d/(d+1)})$. In words, the platform should ensure that the excess supply scales up as the $d/(d+1)$-th power of the load factor $n$ to minimize the total cost; this choice causes the system to be in the so-called quality-and-efficiency driven (QED) regime (i.e., the regime where two types of costs have the same scaling; the QED regime was first identified by Halfin and Whitt \cite{halfin1981heavy}).

This model for $d=2$ is a close cousin of that of Besbes, Castro and Lobel \cite{besbes2018spatial}. The goal of \cite{besbes2018spatial} was the same, namely, to quantify the (scaling of) the optimal amount of excess supply needed, and motivated by ridehailing, they focus on\footnote{A difference between the two models is that \cite{besbes2018spatial} incorporate pickup times (proportional to the match distance) in their formulation, whereas we don't.} $d=2$. An important distinction between the models is that whereas we concretely deal with the spatial supply state resulting from past matching decisions and the consequent match cost, \cite{besbes2018spatial} study an optimistic ``reduced form'' model which simply assumes that a crudely estimated nearest-neighbor-distance is achievable. Like us, \cite{besbes2018spatial} finds that the service firm should use excess supply which scales as the offered load to the exponent $d/(d+1)=2/3$. Thus, our finding above can be viewed as verifying the ``prediction'' of \cite{besbes2018spatial} regarding the optimal scaling of excess supply and extending it to all $d \geq 2$, by showing that the nearest-neighbor-distance is indeed achievable.


Note that for $d=1$, repeating our analysis above leveraging Theorem~\ref{thm:1d-impossibility} and Theorem~\ref{thm:MG-performance} reveals that the optimal scaling of excess supply is between $\sqrt{n}\sqrt{\log n}$ and $\sqrt{n}\log n$.
In contrast, applying the nearest-neigbor reasoning of \cite{besbes2018spatial} would lead to a ``predicted'' optimal excess supply which scales as $\sqrt{n}$, which would be \emph{nearly} correct, but would miss an additional polylogarithmic factor. We emphasize that we consider our findings that the nearest-neighbor-distance is nearly achievable for all $d \geq 1$ in the fully dynamic setting to be a pleasant surprise in light of the fact that nearest-neighbor based reasoning leads us badly astray in the case of 1-dimensional static matching (Fact~\ref{fact:AKT-1d-is-wrong}).


\section{Discussion}
\label{sec:discussion}

Note that the Hierarchical Greedy algorithm which powers our achievability results has the attractive feature that it does not require prior knowledge of the horizon (number of demand units).\footnote{The same holds for the gravitational matching algorithm we adopt from \cite{holden2021gravitational}.}

We mention that our achievability results for \( d\geq 3 \) in both the semi-dynamic and the fully dynamic model can be obtained from the results of Talagrand on static spatial matching \cite{talagrand1992matching}. In these cases, the nearest neighbor distance can be achieved to within a constant factor, and one can show this using \cite{talagrand1992matching} as follows.   Given \( n \) supplies and one demand chosen independently and uniformly from \([0,1]^d\), we imagine generating \( n-1 \) additional demands also uniformly and independently from \([0, 1]^d\), compute a min-cost matching with these supplies and demands, and decide to match the one ``real'' demand to the supply unit it was matched to in this augmented matching. On one hand, the cost of the edge matched is the average cost of an edge in the Talagrand setting, i.e., no more than \( C/n^{1/d} \) for some $C< \infty$ which does not depend on $n$. But also, as the marginal distribution of the matched supply given the set of supplies is uniform over the $n$ points, the remaining $n-1$ points are still independent and uniform over \([0, 1]^d\), so this can be iterated. This argument yields a \( (1/N)\sum_{n=M+1}^{N+M} C/n^{1/d} \leq  C'/N^{1/d} \) expected cost per match in the semi-dynamic setting, cf. Theorem~\ref{thm:semi-dynamic}. In the fully dynamic setting it yields a bound of \( C/m^{1/d} \), cf. Theorem~\ref{thm:MG-performance}. 

In the rest of this section we will discuss related work, and point out some open directions.

\subsection{Related work}

This work builds on and is inspired by the deep and beautiful results on static spatial matching of Ajtai, Komlos \& Tusnady \cite{ajtai1984optimal} in two dimensions, and Talagrand \cite{talagrand1992matching} in three and higher dimensions, and the work of Shor on average-case dynamic bin packing \cite{shor1986average,shor1991pack}, among others. In the average-case dynamic bin packing problem \cite{shor1986average,shor1991pack}, an unknown number of items of size i.i.d. Uniform$(0,1)$ arrive sequentially and must be packed in bins of size $1$, so as to minimize the expected wasted space; Shor shows that this problem is related to static spatial matching in two dimensions, with  time in the bin packing problem being analogous to a spatial dimension in the latter problem. \cite{shor1986average} inspires our proof of the lower bound for our fully dynamic model for $d=1$, while \cite{shor1991pack} bears a spiritual resemblance to our Hierarchical Greedy-based approach to achievability.

There is a rich line of work on (static) matching between a Poisson process (or other translation invariant point processes) and the Lebesgue measure in $\mathbb{R}^d$.  Interest in this infinite setup originated from the fact that an allocation rule
gives rise to a shift-coupling between a point process and its Palm version; see, e.g., \cite{holroyd2005extra}. \cite{hoffman2006stable} study the so-called stable matching between the two measures and prove lower bounds on the typical allocation distance. \cite{chatterjee2010gravitational} study gravitational allocation in the same setting for $d \geq 3$ and show exponential decay of the ``allocation diameter''; subsequently \cite{marko2016poisson} obtained a Poisson allocation with optimal tail behavior for $d \geq 3$ using a different approach inspired by \cite{ajtai1984optimal}. Recently, Holden, Peres, and Zhai \cite{holden2021gravitational} study gravitational matching for uniform points on the surface of a sphere and recover the achievability result of \cite{ajtai1984optimal} for 2-dimensional static spatial matching. Recall that we use a dynamic implementation of \cite{holden2021gravitational} to obtain a tight upper bound in our semi-dynamic model for $d=2$ and $M \leq N^{1-\epsilon}$.

We remark that while a dynamic implementation of gravitational matching (or other approaches developed previously for static matching) may yield tight achievability results in our (easier) semi-dynamic setting, such an approach would yield inferior results in our fully dynamic model for $d=1$ (by a polynomial factor) and $d=2$ (by a polylogarithmic factor), forcing us to develop a novel algorithmic and analytical methodology.\footnote{The idea is to construct a matching (transport) between $m$ i.i.d. ``supply'' points in $[0,1]^d$ and the $m$-times the Lebesgue measure on $[0,1]^d$, and to match each demand arrival to one of the $m$ points based on this transport. Such an approach ensures that the locations of the $m$ points remains i.i.d. uniform. However, for $d=1$ the cost of any transport between $m$ i.i.d. points and the Lebesgue measure is $\Omega(1/\sqrt{m})$. We achieve a much better scaling of $(\log m)^2/m$ via Hierarchical Greedy.} Specifically, our Hierarchical Greedy-based approach takes advantage of averaging over time, which static matching-based approaches fail to do.

Recently, Besbes, Castro and Lobel \cite{besbes2018spatial} estimate the scaling for the amount of ``safety'' supply (the excess of supply over demand) in order to be in the ``quality and efficiency driven (QED) regime''; see the discussion in Section~\ref{sec:analysis}. 

We briefly mention the earlier work of Bertsimas and Van Ryzin \cite{bertsimas1993stochastic} which studies the problem of routing vehicles in the plane under stochastic and dynamic vehicle arrivals, and allows for capacitated vehicles, i.e., each vehicle must return to its ``depot'' after serving a certain maximum number of demand units. They characterize system stability and  the minimum achievable expected waiting time; their results on the scaling of the expected waiting time are driven by the fact that the nearest-neighbor-distance in two dimensions scales as $1/\sqrt{\textup{Density of points}}$.

The very recent work \cite{akbarpour2021value} considers the semi-dynamic model in one dimension $d=1$ (the setting we termed the ``aberrant case'') and observes that excess supply significantly reduces the matching distance. 
\cite{akbarpour2021value} studies only the specific case $M = \Theta(N)$ and $d=1$ and shows an upper bound of $O(\log^3 N/N)$ for the average match distance under the greedy algorithm.  Our hierachical analysis leading to tight bounds for all $d$ and $M$ -- e.g., our results imply that the minimum achievable cost in the special case studied in \cite{akbarpour2021value} is in fact $\Theta(1/N)$ -- is unrelated to the analysis in \cite{akbarpour2021value}.\footnote{The technical contribution of \cite{akbarpour2021value} appears to be an analysis of the \emph{greedy} algorithm in particular (for $d=1$ and $M = \Theta(N)$), since an alternative binning-based approach which partitions $[0,1]$ into subintervals of length $\Theta(\log N/N)$ each, can be immediately shown to have cost upper bounded by $O(\log N/N)$, which is $(\log N)^2$ smaller than their upper bound on greedy. 
The greedy algorithm indeed appears hard to analyze sharply, which motivated us to introduce a hierarchical variant of greedy.}  \cite{akbarpour2021value} do not study the fully dynamic model, which may be setting closest to applications such as ridehailing. Our analysis reveals that, in fact, the fully dynamic model in $d$ dimensions loosely resembles the semi-dynamic model in $d+1$ dimensions, i.e., there is no analog in the fully dynamic model of the $d=1$ case in the semi-dynamic model. Moreover, in the semi-dynamic model the cost in the aberrant $d=1$ case is driven by the longer length scales which makes this case fundamentally different from $d\geq 2$, and hence the aberrant case is unrelated to the fully dynamic model for any $d \geq 1$.

\subsection{Open directions} We leave it as a challenging open problem to close the logarithmic factor gap between our upper and lower bounds for $d =1$ in the fully dynamic model. We expect that the resolution of this open problem will lead to new algorithmic insights, 
specifically, on how to take advantage of ``cancellation'' between mismatches in the quantities of demand and supply at different length scales. We conjecture that our lower bound is tight, and moreover expect that greedy (or any similar approach) will \emph{not} achieve the optimal scaling because it does not leverage such cancellation.

Simulation studies indicate that standard greedy performs slightly better than Hierarchical Greedy in the settings we consider. However, formally analyzing standard greedy seems very challenging and remains open, with the exception of the semi-dynamic model with $d=1$ and $M \propto N$, where a paper \cite{balkanski2023power} subsequent to the present work has shown that standard greedy achieves the optimal scaling of $\Theta(1/N)$.  

Our achievability results relying on Hierarchical Greedy need the match cost to not grow too fast with match distance (here we assumed match cost equal to the match distance). Suppose match cost is equal to the $\textup{(Euclidean distance)}^p$ for $p > 1$. First consider the semi-dynamic model with $d \geq 3$ and $M=0$. For $p < d/2$ our approach using Hierarchical Greedy yields an upper bound of $O(1/N^{p/d})$ on the expected average match cost (the micro length scale dominates the cost) which matches the nearest-neighbor-based lower bound. In contrast for $p>d/2$, the expected cost of Hierarchical Greedy scales as $O(1/\sqrt{N})$ (here the macro length scale dominates the cost), whereas we conjecture that the nearest-neighbor-based lower bound $\Omega(1/N^{p/d})$ is tight for $p < d$. Next, consider the fully-dynamic model with $d=1$. For any $p > 1$, we expect our Hierarchical Greedy-based approach will only yield an upper bound of $O(1/{m})$, whereas expected average match cost of $O(\textup{polylog}(m)/m^{p})$ may be achievable by a different algorithm. Establishing these conjectures would require a novel approach(es), to overcome the deficiency of Hierarchical Greedy that it matches some demand units over a large distance. A related direction is to obtain tail bounds on the matching distance in our dynamic matching models.


It is of interest to explore what happens when demand and supply have different spatial distributions. In this case, the expected average matching cost must exceed the minimum average cost achievable in the continuum limit, and one might hope that the excess cost on top of the cost in the continuum limit may have the same scaling as the characterizations in the present paper. This direction may lead, moreover, to the development of a theory of network revenue management \cite{talluri2004theory} with many types.



\begin{acks}[Acknowledgments]
This paper has  benefited from extensive conversations with Itai Ashlagi, Omar Besbes, Yeon-Koo Che, David Goldberg, Akshit Kumar, Jacob Leshno, Amin Saberi, Michel Talagrand, and
Shawn (Shangzhou) Xia. Above all, we are extremely grateful to the anonymous referee whose astute observations and suggestions have vastly improved the paper. 
\end{acks}

\bibliographystyle{IEEEtran}

\bibliography{matching}

\begin{appendix}

\section*{Proofs of Proposition~\ref{prop:closest-nhbr-lower-bound}, Theorem~\ref{thm:static-matching} and Theorem~\ref{thm:1d-impossibility}}

\begin{proof}[Proof of Proposition~\ref{prop:closest-nhbr-lower-bound}]
The state of the system consists of the locations of the $m$ supply units prior to a period. For arbitrary state, we will show that the expected distance between the next demand arrival and the closest supply unit is at least $1/(Ck^{1/d})$ for some $C= C(d) < \infty$ which does not depend on the state. As a result, the expected match cost in each period is at least this much, establishing the proposition.  

Fix a state. Let $B_i$ be the ball of radius $r=2/(Ck^{1/d})$ around the $i$-th supply unit, for $1\leq i \leq m$. The volume of $B_i$ is proportional to $r^d = \frac{2^d}{C^d m}$, and hence for large enough $C< \infty$, the $\textup{Vol}(B_i) \leq \frac{1}{2m}$. Hence the $\textup{Vol}(\cup_{i=1}^m B_i) \leq  1/2$. On the other hand, $\Vol(\cC) =1$ and hence $\Vol(\cC \backslash (\cup_{i=1}^m B_i)) \geq 1/2$. It follows that with probability at least $1/2$, the demand arrival will be outside this union of balls $\cup_{i=1}^m B_i$, i.e., the distance to the nearest supply unit will be at least $r$. We immediately infer that the expected distance between the next demand arrival and the closest supply unit is at least $r/2=1/(Ck^{1/d})$.
\end{proof}

\begin{proof}[Proof of Theorem~\ref{thm:static-matching}]
All our upper bounds (i.e., achievability results) in this setting are immediate from Theorem~\ref{thm:semi-dynamic} since the semi-dynamic setting is only harder than the static setting due to uncertainty about the locations of future arrivals.
Hence, we only need to prove the lower bounds.

The lower bound for $d \geq 3$ 
follows from the fact that the nearest-neighbor-distance remains $\Omega(1/N^{1/d})$ for any $M \leq N$.

The lower bound for $d=1$ is proved by building on the intuition we provide next.
Observe that (i) the density of demand units is $N$ and the density of supply units is $N+M$, (ii) the stochastic fluctuations in the number of demand (supply) units in an interval of length $L$ are of order $\sqrt{NL}$. As a result, for $M \in [\sqrt{N}, N]$, the excess supply density $M$ ensures that there is no supply shortage over subintervals of length $L$ exceeding $L_*\sim N/M^2$ since the excess supply $ML$ is larger than the stochastic fluctuation $\sqrt{NL}$ for such intervals.  Now, the supply shortage over a subinterval of length $L \lesssim L_*$ requires matching over an expected distance $\Delta_L \sim \sqrt{L/N}$ for $L \lesssim L_*$, since we must have $N\Delta_L \sim \sqrt{NL} \Rightarrow \Delta_L \sim \sqrt{L/N}$.
The matching cost is dominated by that from scale $L \sim L_*$ (provided $N/M^2 > 1/N \Leftrightarrow M<N$, where $1/N$ is the nearest-neighbour distance), which is $\Delta_{L_*} \sim 1/M$.

We now turn this intuition into a formal proof. Fix $M \in [\sqrt{N}, N]$. Define $L_0 \triangleq N/(2M^2) \leq 1/2$. For any matching $\pi$ between supply and demand which matches all $N$ demand units, define $f: [0,1] \to \{0, 1, 2, \dots \}$ as $$f(x) \triangleq \{\# \textup{ matched pairs s.t. the members of the pair are located on opposite sides of } x  \}\, .$$ The distance for any individual match may be expressed as
$$\int_{0}^1 \mathbb{I}(\textup{the members of the pair are located on opposite sides of } x)dx \, .$$
and, summing over matched pairs and dividing by $N$, the average matching distance across all pairs is clearly
$(1/N)\int_{0}^1 f(x) dx$. We now establish the lower bound
\begin{align}
\mathbb{E} \Big[ \min_\pi \int_{0}^1 f(x) dx \Big] \geq N/(C_1 M) \, ,
\label{eq:int-fx}
\end{align}
and dividing by $N$ on both sides this will immediately imply the desired lower bound $1/(C_1 M)$ on the expected average matching distance. 

Consider any subinterval $[x, x+L_0]$ of length $L_0$. The number of demand units in the subinterval is Binomial$(N, L_0)$ which has mean $NL_0 = N^2/(2M^2)$ and standard deviation $\sqrt{NL_0 (1-L_0} \geq  N/(2M)$. The number of supply units is a Binomial with mean $N^2/(2M^2) + N/(2M)$ and standard deviation less than $\sqrt{(N+M)L_0(1-L_0)} \geq  N/(2M)$. As a result the excess of demand over supply in the subinterval (defined as 0 if the demand is weakly less than the supply) has expected value at least $4N/(C_1 M)$ for some $C_1 < \infty$. (With probability at least $4/C_1$, the demand exceeds supply by at least $N/M$.) All this excess demand must be matched to supply outside the subinterval under any matching $\pi$ which matches all the demand. Hence $\Ex \big [\min_{\pi} (f(x) + f(x+L_0)) \big ] \geq 2N/(C_1 M)$. Integrating the left hand side over $x \in [0, 1-L_0]$, we have
\begin{align}
  \Ex \Big [\int_{0}^{1-L_0}\min_{\pi}(f(x) + f(x+L_0)) dx \Big] \geq (1-L_0)4N/(C_1 M) \geq 2N/(C_1 M) \, ,
\end{align}
using $L_0 \leq 1/2$.
Now, the left-hand side is bounded above by
\begin{align*}
  \Ex \Big [\min_{\pi}  \int_{0}^{1-L_0} ( f(x) + f(x+L_0) ) dx \Big] \leq 2 \int_{0}^{1}  f(x) dx \, .
\end{align*}
Putting the prior inequalities together, we obtain \eqref{eq:int-fx} as required, and this completes the proof.

For $d=2$, the proof of the lower bound of the AKT theorem \cite{ajtai1984optimal} (we refer to the version in  Talagrand's book \cite[Section 4.6]{talagrand2021upper}) extends easily to any $M < N^{1-\epsilon}$ since all length scales $L$ with $\sqrt{NL^2} \gtrsim ML^2 \Leftrightarrow L \lesssim \sqrt{N}/M$, which includes scales $L \in [\sqrt{\log N}/\sqrt{N}, N^\epsilon/\sqrt{N}]$, contribute $\Theta(1/\sqrt{N \log N})$ to the lower bound on cost (additively) and there are at least order $\log_2 \Big (\frac{N^\epsilon}{\sqrt{\log N}}\Big ) = \Omega(\log N)$ such scales for $M < N^{1-\epsilon}$. (Here we used that the smallest length scale is $\sqrt{\frac{\log N}{N}}$ and that consecutive length scales are separated by a factor $2^{\Theta(1)}= 1+\Theta(1)$ from each other in the proof.) 
For $d=2$ and $M \in [\epsilon N, N]$, as for $d\geq 3$, the total number of supply units $M+N$ is no more than $2N$, hence the expected distance to the nearest neighbor is at least $1/(C_2 \sqrt{N})$ which immediately yields the desired lower bound.
\end{proof}

Towards the proof of Theorem~\ref{thm:1d-impossibility}, we will need the following duality lemma for a minimum cost rightward perfect matching between two sets of points located in the unit square. While we will need only weak duality, we provide the full strong duality version.
\begin{lemma}
Let $(X_i)_{i \leq n}$ and $(Y_j)_{j \leq n}$ be located in $[0,1]^2$ such that it is possible to construct a ``rightward'' matching, i.e., a permutation $\pi$ on $(1, 2, \dots, n)$ such that the horizontal($x$) coordinate of $Y_{\pi(i)}$ is strictly larger than the horizontal coordinate of $X_i$, $x(Y_{\pi(i)}) > x(X_i) $ for all $i \leq n$.
Denote the set of allowed pairs by $E \triangleq \{(i,j): Y_{j} \textup{ is strictly to the right of } X_i\}$, the set of allowed matchings by $\Pi_{\textup{rt}} \triangleq \{\pi: (i, \pi(i)) \in E \; \forall i \leq n  \} \neq \emptyset$, and the ``vertical distance'' between $X_i$ and $Y_j$ by
$d^{\textup{v}}(X_i, Y_j) \triangleq |y(X_i) - y(Y_j)|$ where $y(\cdot)$ is the vertical coordinate of a point. Then we have
$$
\inf_{\pi \in \Pi_{\textup{rt}}} \sum_{i \leq n} d^{\textup{v}}(X_i, Y_{\pi(i)}) = \sup_{f \in \cF} \sum_{i \leq n} f(Y_i) - f(X_i) \, ,
$$
where $\cF$ is the set of functions $f$ on $[0,1]^2$ satisfying:
\begin{itemize}
  \item For each $y \in [0,1]$, the function $f(\cdot, y)$ is non-increasing.
  \item For each $x\in [0,1]$, the function $f(x, \cdot)$ is 1-Lipschitz.
\end{itemize}
\label{lemma:rightwardmatching-dual-function}
\end{lemma}
\begin{proof}
We start by establishing ``weak duality'' (LHS $\geq$ RHS). Note that for any $f \in \cF$ and any matching $\pi \in \Pi_{\textup{rt}}$, for all $i \leq n$ we have
\begin{align}
  f(Y_{\pi(i)}) - f(X_i) \leq f(Y_{\pi(i)}) - f(x(Y_{\pi(i)}), y(X_i)) \leq d^{\textup{v}}(X_i, Y_{\pi(i)}) \,
\end{align}
where we used that $f(\cdot, y(X_i))$ is non-increasing to get the first inequality, and that $f(x(Y_{\pi(i)}), \cdot)$ is 1-Lipschitz to get the second inequality. Summing over $i\leq n$ we get
\begin{align*}
\sum_{i \leq n} f(Y_i) - f(X_i) = \sum_{i \leq n} f(Y_{\pi(i)}) - f(X_i) &\leq \sum_{i \leq n} d^{\textup{v}}(X_i, Y_{\pi(i)})\, .
\end{align*}
Since this holds for any $f \in \cF$ and any matching $\pi \in \Pi_{\textup{rt}}$, we have
$$
\inf_{\pi \in \Pi_{\textup{rt}}} \sum_{i \leq n} d^{\textup{v}}(X_i, Y_{\pi(i)}) \geq  \sup_{f \in \cF} \sum_{i \leq n} f(Y_i) - f(X_i) \, .
$$

We now establish the reverse inequality (``strong duality'').  We make use of {\cite[Proposition 4.3.2]{talagrand2021upper}} with
\begin{align*}
  c_{ij}  \triangleq \left \{
  \begin{array}{ll}
  d^{\textup{v}}(X_i, Y_{\pi(i)}) & \textup{for all } (i,j) \in E\\
  \infty & \textup{otherwise}
  \end{array} \, \right .
\end{align*}
which tells us that there exists a dual optimum $(w_i)_{i \leq n}$ and $(w_j')_{j \leq n}$ such that
\begin{align}
  w_i + w_j'  \leq d^{\textup{v}}(X_i, Y_{j}) \ \textup{ for all } (i,j) \in E
  \label{eq:dual-cond-rt-matching}
\end{align}
and
\begin{align}
  \sum_{i\leq n} w_i + w'_i = \inf_{\pi \in \Pi_{\textup{rt}} }\sum_{i\leq n}  d^{\textup{v}}(X_i, Y_{\pi(i)}) \, .
  \label{eq:strong-duality-prop432}
\end{align}
Define $f: [0,1]^2 \rightarrow \mathbb{R}$ by
$$f(X) \triangleq \max_{j: Y_j \in \textup{Rt}(X)} w_j' - d^{\textup{v}}(X, Y_j) \quad \textup{ where } \textup{Rt}(X) \triangleq \{ Y: x(Y)> x(X)\} \, .$$
It is easy to check that this $f \in \cF$.
Note that $w_j' \leq f(Y_j)$ by definition and $w_i \leq - f(X_i)$ using the inequalities \eqref{eq:dual-cond-rt-matching}. It follows that $\sum_i w_i + w'_i \leq \sum_i f(Y_i) - f(X_i)$.  Combining with \eqref{eq:strong-duality-prop432} we have
$\inf_{\pi \in \Pi_{\textup{rt}} }\sum_{i\leq n}  d^{\textup{v}}(X_i, Y_{\pi(i)}) \leq \sum_i f(Y_i) - f(X_i)$ for this $f \in \cF$, which completes the proof.
\end{proof}

 Our proof of Theorem~\ref{thm:1d-impossibility} will draw crucially upon the Ajtai-Komlos-Tusnady \cite{ajtai1984optimal} lower bound for two-dimensional static matching, and thus our proof may be viewed as a formalization of the connection between the fully dynamic model in 1 dimension and the static model in 2 dimensions. Our proof here resembles a proof of Shor \cite[Lemma 1]{shor1986average}.\footnote{There is an analogy between the quantity $m$ here and the quantity $\sqrt{n \log n}$ there, i.e., $n \sim m^2/\log m$.  The ``right-matching'' constraint in \cite[Lemma 1]{shor1986average} is analogous to the obvious constraint in our setting that a demand unit can only match with a supply unit which arrived previously. The expected cost (``vertical distance'') per match there is shown to be at least $\sqrt{\log n/n}$ which evaluates to  $\log m/m$ as the corresponding lower bound on expected match distance in our setting.}, who establishes a lower bound for average-case online bin packing.

\begin{proof}[Proof of Theorem~\ref{thm:1d-impossibility}]

Consider a horizon $N = N_0 = Cm^2/\log m$ for some $C$ which we will choose later. (We will show how to extend to general horizons later.) 
Suppose we had perfect foresight regarding the locations of all supply and demand arrivals, and could choose matches accordingly. (We will show that the claimed lower bound holds even with this additional information.) This turns our matching problem into a 2-dimensional static rightward perfect matching problem with the vertical($y$) axis corresponding to space in the fully dynamic model, and the horizontal($x$) axis corresponding to scaled time $t/N$ in the fully dynamic model. The 2-d rightward static matching problem is specified as follows (parentheses will capture the corresponding features in the fully dynamic model, towards readability):
\begin{itemize}
\item (Total number of supply and demand units) There are $n = N+m-1$ supply units whose locations we will denote by $(X_i)_{i \leq n}$, and an equal number of demand units located at $(Y_i)_{i \leq n}$, in the unit square.
\item (Initial supply inventory) There are $m$ supply units with horizontal coordinate $x=0$ (i.e., present at time $t=0$), and arbitrary $y$ coordinates (spatial locations).
\item  (Demand and supply spatial locations are i.i.d. uniform) All other supply and demand units have a $y$-coordinate which is drawn i.i.d. uniformly from $[0,1]$. We will specify their horizontal coordinates next.
\item (One demand and one supply unit arrival in each period) One supply unit and one demand unit are at $x = t/N$ for all $t = 1, 2, \dots, N-1$. There are no supply units and $m$ demand units at $x=N/N=1$ (one is the demand unit which arrives at $t = N$ in the fully dynamic model; the remaining $m-1$ are ``dummy'' demand units we introduce for technical convenience so that there are an equal number of demand and supply units in our static problem).
\item (Each demand unit must be matched immediately, hence to a supply unit which arrived previously) Each demand unit must be matched, to a supply unit with a strictly smaller $x$-coordinate that the demand unit. (Note that it is feasible to construct a matching which satisfies this requirement.)
\item (Match cost) Each matched pair incurs a cost equal to the absolute difference between the $y$-coordinates of the demand unit and the supply unit. The goal is to minimize the total match cost.
\end{itemize}
We will prove that the expected total (vertical) match cost of any feasible perfect matching in this model is at least $(\epsilon_1/3) \sqrt{n \log n}) \geq  (\epsilon_1/3) \sqrt{N \log N}$ for some $\epsilon_1 > 0$ and any $C \geq 36/\epsilon_1^2$, for large enough $m$. We fix $C= 36/\epsilon_1^2$. Then the contribution of the $m-1$ dummy demand nodes to the total cost is bounded above by $m-1 \leq \sqrt{(N \log N)/C})= (\epsilon_1/6)\sqrt{N \log N}$, so that we  obtain an $(\epsilon_1/6) \sqrt{N \log N}$ lower bound on the expected total cost for the $N$ actual demand nodes, i.e., an expected average cost per match  in the fully dynamic model of at least $(\epsilon_1/6) \sqrt{\log N/N} \geq (\epsilon_1/6)\sqrt{(\log m)^2/C m^2} = (\epsilon_1^2/36) \log m/m$  for horizon $N=N_0=Cm^2/\log m$ in the fully dynamic model, for any matching policy. We will later show how to expect this lower bound to arbitrary horizon $N \geq m^2$.
\smallskip

AKT \cite{ajtai1984optimal} showed (as captured in our Theorem~\ref{thm:static-matching}) that for some $\epsilon_1 > 0$, we have
\begin{align}
  \Ex \bigg [\inf_{\tilde{\pi}} \sum_{i \leq n} d(\tilde{X}_i, \tilde{Y}_{\tilde{\pi}(i)}) \bigg ] \geq \epsilon_1 \sqrt{n\log n}
  \label{eq:AKT-lower-bound}
\end{align}
 for $(\tilde{X_i})_{i \leq n}$ and $(\tilde{Y_i})_{i \leq n}$ distributed i.i.d. uniformly in the unit square. 
\eqref{eq:AKT-lower-bound} does not directly enable us to prove the desired lower bound on expected cost because there are some differences between the distributions of points, the constraints, and the definition of cost, in the two models. In the AKT model the points $(\tilde{X_i})_{i \leq n}$ and $(\tilde{Y_i})_{i \leq n}$ are distributed i.i.d. uniformly in the unit square, whereas the $x$-coordinates in our case are deterministic (also the $y$-coordinates of the $m$ supply units at $x=0$ are arbitrary). We have a rightward matching constraint which is absent in AKT, and our cost is the vertical distance, whereas AKT consider the Euclidean distance.

The difference in the horizontal distributions of supply (demand) points in the two models is relatively easy to handle. To move the locations in the AKT model $(\tilde{X_i})_{i \leq n}$ and $(\tilde{Y_i})_{i \leq n}$ to the locations in our model $({X_i})_{i \leq n}$ and $({Y_i})_{i \leq n}$, we index the supply (demand) units from left to right in both models, and then move each AKT supply (demand) unit leftwards or rightwards so that it is located at the $x$-coordinate of the corresponding supply (demand) unit in our model. By the central limit theorem, along with $m = \Theta(\sqrt{n \log n})$ and the fact that the $x$-locations of $({X_i})_{i \leq n}$ and $({Y_i})_{i \leq n}$ deviate from even spacing by at most $m/(n+1)$, the expected distance moved by the $i$-th supply unit in this process, for each $i \leq n$, is bounded as
\begin{align*}
  \Ex [|x(\tilde{X}_i) - x(X_i)|] &\leq m/(n+1) + O(1/\sqrt{n}) \leq \sqrt{(\log n)/(2Cn)} + O(1/\sqrt{n}) \\
  &\leq \sqrt{\log n/(nC)}
\end{align*}
for large enough $m$, and the same bound holds for the difference in demand unit locations. In the case of the $m$ initially present supply units, we also need to move those points vertically (by at most $1$ unit each) to allow for arbitrary $y$-coordinates in our model, i.e.,
$
  |y(\tilde{X}_i) - y(X_i)| \leq 1\ \textup{for all } i \leq m \, ,
$ while all other $y$-coordinates are identical.
 It follows that
 \begin{align}
    \label{eq:move-bounds}
   \Ex \bigg [ \sum_{i \leq n} d(\tilde{X}_i, X_i) \bigg ] &\leq  \sqrt{(n \log n)/C} + m \leq 2 \sqrt{(n \log n)/C} \quad \textup{and} \\
    \Ex \bigg [ \sum_{i \leq n} d(\tilde{Y}_i, Y_i) \bigg ] &\leq  \sqrt{(n \log n)/C} \, .
\nonumber
 \end{align}

AKT \cite{ajtai1984optimal} construct a ``witness'' dual function with large empirical discrepancy to establish their lower bound \eqref{eq:AKT-lower-bound}. We will transform the AKT dual function into a function $f$ which is a dual function for our rightward matching problem,  and inherits (with small slack) the AKT lower bound on the empirical discrepancy.

The  (random) dual function $\tilde{f}: [0,1]^2 \rightarrow \mathbb{R}$ constructed in the lower bound proof of AKT \cite{ajtai1984optimal} satisfies the following properties:
\begin{itemize}
  \item $\tilde{f}$ is 1-Lipschitz.
  \item Large empirical discrepancy:
  \begin{align}
    \Ex \bigg [\sum_{i\leq n} \tilde{f}(\tilde{Y}_i) - \tilde{f}(\tilde{X}_i) \bigg ] \geq \epsilon_1 \sqrt{n \log n} \, .
  \label{eq:tf-lb}
  \end{align}
\end{itemize}
Their lower bound \eqref{eq:AKT-lower-bound} then follows since $\tilde{f}$ being 1-Lipschitz implies that $d(\tilde{X}_i, \tilde{Y}_{\tilde{\pi}(i)}) \geq \tilde{f}(\tilde{Y}_{\tilde{\pi}(i)}) - \tilde{f}(\tilde{X}_{i})$, which leads to $\Ex \big [ \inf_{\tilde{\pi}} \sum_{i \leq n} d(\tilde{X}_i, \tilde{Y}_{\tilde{\pi}(i)} \big ] \geq \Ex \big [\sum_{i\leq n} \tilde{f}(\tilde{Y}_i) - \tilde{f}(\tilde{X}_i)\big ] \geq \epsilon_1\sqrt{n \log n}$. 

We define $f(x,y) \triangleq  \tilde{f}(x,y)-x$. Here we subtract $x$ to ensure that $f(\cdot, y)$ is non-increasing for all $y \in [0,1]$. Since $\tilde{f}$ is 1-Lipschitz in $[0,1]^2$, it follows that $f(x, \cdot)$ is 1-Lipschitz for all $x \in [0,1]$. With these two properties, we know from Lemma~\ref{lemma:rightwardmatching-dual-function} that $f(\cdot)$ is a dual function for our rightward matching problem, and weak duality tells us that:
$$
\inf_{\pi \in \Pi_{\textup{rt}}} \sum_{i \leq n} d^{\textup{v}}(X_i, Y_{\pi(i)}) \geq  \sum_{i \leq n} f(Y_i) - f(X_i) =  \sum_{i \leq n} \tilde{f}(Y_i) - \tilde{f}(X_i) - x(Y_i) + x(X_i) \, .
$$
By definition of horizontal coordinates in our problem $ \sum_{i \leq n} x(Y_i) - x(X_i) = m$.
Taking expectations in the previous inequality, we obtain
\begin{align}
\Ex \bigg [ \inf_{\pi \in \Pi_{\textup{rt}}} \sum_{i \leq n} d^{\textup{v}}(X_i, Y_{\pi(i)}) \bigg ] \geq -m+ \Ex \bigg [ \sum_{i \leq n} \tilde{f}(Y_i) - \tilde{f}(X_i) \bigg ]\, .
\label{eq:rt-wd}
\end{align}
Now, since $\tilde{f}$ is 1-Lipschitz, we know that
$|\tilde{f}(\tilde{Y}_i)-\tilde{f}(Y_i)| \leq  d(\tilde{Y}_i,Y_i)$ and $|\tilde{f}(\tilde{X}_i)-\tilde{f}(X_i)| \leq  d(\tilde{X}_i,X_i)$. It follows that
\begin{align*}
\sum_{i \leq n} \tilde{f}(Y_i) - \tilde{f}(X_i) \geq \sum_{i \leq n} \big( \tilde{f}(\tilde{Y}_i) - \tilde{f}(\tilde{X}_i) \big ) - \sum_{i \leq n}d(\tilde{Y}_i,Y_i) - \sum_{i \leq n} d(\tilde{X}_i,X_i)\, .
\end{align*}
Plugging back into \eqref{eq:rt-wd} and leveraging \eqref{eq:tf-lb}, \eqref{eq:move-bounds} and $m \leq \sqrt{(n \log n)/C}$, we obtain
\begin{align*}
  \Ex \bigg [ \inf_{\pi \in \Pi_{\textup{rt}}} \sum_{i \leq n} d^{\textup{v}}(X_i, Y_{\pi(i)}) \bigg ]
  \geq \epsilon_1 \sqrt{n \log n} - 4 \sqrt{(n \log n)/C} \geq (\epsilon_1/3)  \sqrt{n \log n}\, ,
\end{align*}
where we ensure the second inequality holds by choosing $C \geq 36/\epsilon_1^2$.

\smallskip
As explained previously, this gives us a lower bound of $(\epsilon_1^2/36) \log m /m$ on the expected cost per match in the fully dynamic model for $N=N_0$ and large enough $m$. We can extend to any horizon $N \geq N_0$ by dividing the horizon into epochs of duration $N_0$ periods each. For an epoch of duration $N_0$, for arbitrary initial state at the beginning of the epoch, the aforementioned argument shows that the expected average cost in the epoch is at least $(\epsilon_1^2/36)\epsilon \log m /m$ for any matching policy. Even if the last epoch of length $N_0$ extends past period $N$ (because $N$ is not divisible by $N_0$) we immediately deduce a lower bound of $(\epsilon_1^2/72) \log m /m$ on the expected cost per match over the horizon $N$, for large enough $m$.
A lower bound of $\epsilon \log m /m$ for arbitrary $N \geq m^2$ and $m \geq 2$ follows, for $\epsilon \in(0, \epsilon_1^2/72 ]$ chosen to ensure that the lower bound holds for all $m \geq 2$. 
\end{proof}
\end{appendix}




%





\end{document}